\documentclass[12pt]{article}

\usepackage[english]{babel}
\usepackage{amsfonts,amsmath, amssymb,amsthm}
\usepackage{dutchcal}
\usepackage{hyperref}
\usepackage{csquotes}

\usepackage{xcolor}

\usepackage{enumerate}
\usepackage{epsf,epsfig,amsfonts,a4wide}
\usepackage{amsfonts, mathdots}
 
\usepackage{url}

\usepackage{graphicx}

\newtheorem{Theorem}{Theorem}[section]

\newtheorem{Lemma}{Lemma}[section]
\newtheorem{Proposition}{Proposition}[section]
\newtheorem{Corollary}{Corollary}[section]

\newtheorem{Ex}{Example}[section]
\newtheorem{Remark}{Remark}[section]

\theoremstyle{remark}

\newcommand{\be}{\begin{equation}}
\newcommand{\ee}{\end{equation}}

\newcommand{\R}{\mathbb{R}}\newcommand{\Id}{\textrm{\rm Id}}

\newcommand{\gl}{\mathrm{gl}}

\newcommand{\ddd}{\mathrm{d}}

\newcommand{\pd}[2]{\frac{\partial#1}{\partial#2}}

\newcommand{\dd}{{\mathrm d}\,}

\newtheorem{maintheorem}{\normalfont\textsc{Theorem}}

\newcommand{\weg}[1]{}

\usepackage[
  backend=biber,
  style=numeric,
  sorting=nyt,
  isbn=false, doi=true, url=false  
]{biblatex}
\title{Local description of $\gl$-regular Haantjes operators}
\author{Alexey V.\ Bolsinov\footnote{School of Mathematics,
 Loughborough University,
 LE11 3TU, UK  and Institute of Mathematics and Mathematical Modeling, Almaty, Kazakhstan\ \ \quad {\tt A.Bolsinov@lboro.ac.uk}},    
Andrey Yu.\  Konyaev\footnote{Faculty of Mechanics and Mathematics and Center for Fundamental and Applied Mathematics, Moscow State University, 119992, Moscow, Russia 
 \ \ \quad{\tt  maodzund@yandex.ru}}  \, and 
    Vladimir S.\ Matveev\footnote{ Corresponding author. 
Institut f\"ur Mathematik, Friedrich Schiller Universit\"at Jena,
07737 Jena,  Germany  \ \ \quad {\tt  vladimir.matveev@uni-jena.de}}}
\date{} 
\begin{document}

\maketitle

\begin{abstract} We study Haantjes operators, that is, $(1,1)$-tensor fields with vanishing Haantjes torsion. Our main result is a complete local description of $\gl$-regular Haantjes operators. Additional results include a splitting theorem for general (not necessarily $\gl$-regular) Haantjes operators and, more generally, for operators with vanishing generalised Nijenhuis torsion of an arbitrary level, as well as a complete treatment and understanding of the case when the eigenvalues of a Haantjes operator are complex; the latter case was ignored in many previous papers on this and related topics. 
\end{abstract}
\tableofcontents

{
\section{Introduction }

\subsection{Definitions and main results}
By an {\it operator} ({\it field}) we understand a $(1,1)$-tensor field on an open set $U\subseteq \mathbb{R}^n$. Recall that the {\it Nijenhuis torsion} of an operator field $L$ is the $(1,2)$-tensor field given by
\begin{equation}\label{nij_m}
\mathcal N_L (\xi, \eta) =  L^2[\xi, \eta] + [L\xi, L\eta] - L\,[L\xi, \eta] - L\,[\xi, L\eta],
\end{equation}
and the {\it Haantjes torsion} of $L$ is the $(1,2)$-tensor field given by
\begin{equation}\label{haa_m}
\mathcal H_L (\xi, \eta) =  L^2 \mathcal N_L(\xi, \eta) + \mathcal N_L (L\xi, L\eta) - L \,\mathcal N_L (L\xi, \eta) - L\, \mathcal N_L (\xi, L\eta).
\end{equation}
The square bracket $[\xi, \eta]$ in \eqref{nij_m} denotes the Lie bracket of vector fields $\xi$ and $\eta$. We say that  $L$ is a {\it Nijenhuis operator}  if its Nijenhuis torsion  $\mathcal N_L$ vanishes, and a {\it Haantjes operator}, if its Haantjes torsion $\mathcal H_L$ vanishes.

We say that an operator $L$ is {\it $\gl$-regular} at a point $\mathsf p\in U$ if, at this point, the geometric multiplicity of each of its eigenvalues equals $1$. We say that $\mathsf p\in U$ is {\it algebraically generic},  if there exists a neighbourhood of $\mathsf p$ in which the Segre characteristic of $L$, i.e., the sizes of the Jordan blocks related to each eigenvalue $\lambda_i$ of $L$,  does not change (specific values of $\lambda_i$'s are not important here; they can vary from point to point). Clearly, almost every point is algebraically generic.

Our main result, presented below in two equivalent Theorems \ref{thm:A} and \ref{thm:B}, is a local description of $\gl$-regular Haantjes operators near algebraically generic points.

\begin{maintheorem}\label{thm:A} {\it  
Let $L$ be a $\gl$-regular Haantjes operator. Then, in a neighbourhood of an algebraically generic point, $L$ can be written as
\begin{equation} 
\label{eq:polynom}
L=p(M)= f_0(x) \Id + f_1(x)M + f_2(x)M^2 + \dots + f_{n-1}(x)M^{n-1},
\end{equation}
where $f_i(x)$ are smooth functions  and $M$ is a Nijenhuis operator. Moreover, the operator $M=\bigl( M^i_j\bigr)$ can be chosen to be constant in a suitable coordinate system.}
\end{maintheorem}

As we recall below, for any Nijenhuis operator $M$ and any functions $f_0,\dots, f_{n-1}$, the operator $p(M)$ given by \eqref{eq:polynom} is Haantjes. It is easy to see that if
$L=p(M)$ is $\gl$-regular, then $M$ is $\gl$-regular automatically. Conversely, for a generic polynomial $p(\cdot)$ and a $\gl$-regular operator $M$, the operator $L=p(M)$ is $\gl$-regular also. More precisely, if $M$ is $\gl$-regular, and for the eigenvalues $\lambda_i$'s of $M$ we have $p'(\lambda_i)\ne 0$ and $p(\lambda_i)\ne p(\lambda_j)$ for $\lambda_i\ne \lambda_j$, then $p(M)$ is $\gl$-regular.
In view of this, Theorem \ref{thm:A} gives a complete description of $\gl$-regular Haantjes operators at almost every point.

\begin{Ex} \label{Ex:toeplitz1}{\rm 
   Take an operator field $M$ whose matrix, in local coordinates, is the standard upper triangular Jordan block $J_n(\lambda)$ with a constant eigenvalue $\lambda\in\R$. Then the operator $L=p(M)$ given by \eqref{eq:polynom} takes the form
       \begin{equation}\label{eq1}
    L = L_{\mathrm{Toeplitz}}=\begin{pmatrix}
     g_n & g_{n - 1} & \!\!\!\dots & g_2 & \!\! g_1  \\
     0 &  g_n & \!\!\!g_{n-1} & & \!\! g_2 \\
   0  &  0 &\!\!\! \ddots  & \ddots   &\!\!  \vdots \\
    \vdots & \vdots  &\!\!\! \ddots & g_n   & \!\! g_{n - 1}\\
     0 & 0 & \!\!\! \dots  & 0 &\!\!  g_n
    \end{pmatrix},
    \end{equation}
where $g_i$ are certain linear combinations of the functions $f_i$. The formula connecting $f_i$ and $g_i$ is invertible, so for any functions $g_i$ one can construct the unique polynomial $p$ of degree $\le n-1$, with functional coefficients, such that $L_{\mathrm{Toeplitz}}=p(M)$.
The form \eqref{eq1} is called the (real) {\it Toeplitz form} \cite{bottcher}.
}\end{Ex}

\begin{Ex} \label{Ex:toeplitz2}{\rm
In dimension $2n$,    take
   \begin{equation}
\label{eq:CompJblock_m}
M = J_n(a\pm\mathrm{i}\, b)=
\begin{pmatrix}
\Lambda & \Id_2 &  0  & \dots & 0  \\
& \Lambda & \Id_2&\ddots  & \vdots\\
& & \ddots & \ddots &0 \\
& & &\Lambda  &\Id_2 \\
& & & & \Lambda \\
\end{pmatrix}
\end{equation}
where each entry is a $2\times 2$ block representing a complex number, namely $\Lambda=\begin{pmatrix} a & - b \\ b & a
\end{pmatrix}\simeq a+\mathrm{i}\, b$ and $\Id_2 = \begin{pmatrix} 1 & 0 \\ 0 & 1\end{pmatrix} \simeq 1 + \mathrm{i} \, 0 = 1$. We will refer to \eqref{eq:CompJblock_m} as a real Jordan block corresponding to a pair of complex conjugate eigenvalues $\lambda^{\pm}=a\pm\mathrm{i}\, b$, $b\ne 0$.

   Then the operator $L$ given by \eqref{eq:polynom} is as follows
\begin{equation}\label{eq1b_m}
    L = \begin{pmatrix}
     G_n & G_{n - 1} & \!\!\!\dots & G_2 & \!\! G_1  \\
     0 &  G_n & \!\!\!G_{n-1} & & \!\! G_2 \\
   0  &  0 &\!\!\! \ddots  & \ddots   &\!\!  \vdots \\
    \vdots & \vdots  &\!\!\! \ddots & G_n   & \!\! G_{n - 1}\\
     0 & 0 & \!\!\! \dots  & 0 &\!\!  G_n
    \end{pmatrix} \end{equation} 
with $G_j$ being a $2\times 2$ block of the form $G_j = \begin{pmatrix} u_j  &  - v_j \\ v_j & u_j  \end{pmatrix} \simeq g_j=u_j + \mathrm{i} \, v_j$, where $u_j$ and $v_j$ are functions constructed from the functions $f_i$, $i=0,\dots, 2n-1$, by certain linear formulas. As in Example \ref{Ex:toeplitz1}, these formulas are invertible, so every operator field of the form \eqref{eq1b_m} is $p(M)$ for a suitable polynomial $p$. We refer to \eqref{eq1b_m} as {\it complex Toeplitz form}.
}\end{Ex}

\begin{maintheorem}\label{thm:B} {\it 
Let $L$ be a $\gl$-regular Haantjes operator. Then, in a neighbourhood of an algebraically generic point, there exists a coordinate system such that $L$ is block diagonal, \begin{equation} \label{eq:L} L= 
\operatorname{diag}\bigl(L_1,\dots, L_k\bigr)=  \begin{pmatrix} L_1 && \\ & \ddots &  \\ && L_k \end{pmatrix},  \end{equation}  and each block  $L_i$ is in real or complex Toeplitz form.
Moreover, any $L$ given by \eqref{eq:L}  with real or complex Toeplitz blocks is a Haantjes operator.} 
\end{maintheorem}

We emphasize that the functions $g_j$, $u_j$, and $v_j$ appearing in the blocks may depend on all variables and the functions $u_j+\mathrm i\, v_j$ are not necessarily holomorphic.

\subsection{Motivation}

Theorems \ref{thm:A} and \ref{thm:B} answer natural and important structural questions about the behaviour of $(1,1)$-tensor fields. They go  back essentially to J.\,Schouten \cite{schouten1}, who explicitly posed the following diagonalisability problem. Assume that an operator field $L$  on $U\subset\R^n$
has $n$ pairwise distinct real eigenvalues. When does there exist a coordinate system $u^1, \dots, u^n$ in which $L$ is diagonal? This particular question was answered by A.\,Nijenhuis and J.\,Haantjes, both PhD descendants of J.\,Schouten. Namely, A.\,Nijenhuis in his seminal work \cite{nijenhuis1951} introduced the Nijenhuis torsion \eqref{nij_m}, and wrote a condition for diagonalisability of $L$ in coordinate form. The tensorial version of this condition is due to J.\,Haantjes \cite{haantjes}. It also covers the case of operator fields $L$ that are pointwise diagonalisable but not necessarily $\gl$-regular: for such operator fields, the local diagonalisability at algebraically generic points is equivalent to the vanishing of the Haantjes torsion \eqref{haa_m}.   However, for non-diagonalisable operators the geometric meaning of the Haantjes condition has remained obscure for 70 years since the pioneering work of J.\,Haantjes \cite{haantjes} published in 1955.

Vanishing of the Nijenhuis and Haantjes torsions are the simplest differential-geometric conditions on operator fields (see e.g. \cite{puninstkii, kolar}).  In some sense, they are analogous to such well-known conditions in differential geometry as the Jacobi identity for a Poisson bivector $P= \sum P^{ij}\partial_{x^i}\wedge\partial_{x^j}$ and closedness for  a symplectic form $\omega = \sum \omega_{ij}\dd x^i\wedge \dd x^j$. Nijenhuis and Haantjes operators naturally pop up in many a priori unrelated branches of mathematics and mathematical physics. The general philosophy of treating them as primary objects,  as outlined in the program paper \cite{nij1} on Nijenhuis geometry, has proven very fruitful. This  is demonstrated by a series of subsequent publications on this topic (entitled {\it Nijenhuis Geometry} II--IV and {\it Applications of Nijenhuis Geometry} I--V), some of which solved explicitly posed open problems with no apparent connection to Nijenhuis geometry.  

In discussions with our colleagues, the following questions were frequently raised: is the vanishing of $\mathcal N_L$ (Nijenhuis condition) too strong? Would it make sense to replace it by the weaker Haantjes condition $\mathcal H_L=0$ to expand the range of possible applications? In many cases, this question came from experts in the theory of finite- and infinite-dimensional integrable systems. Indeed, the Haantjes condition $\mathcal H_L=0$ naturally appeared in different studies and in many setups. Let us give some classic and new examples. 

Consider an $n$-dimensional metric $g$ such that its geodesic flow (i.e., the Hamiltonian system with the Hamiltonian  function $H = \tfrac{1}{2}g^{\alpha\beta} p_\alpha p_\beta$) admits $n$ Poisson-commuting, functionally independent quadratic integrals of the form $I_i= g^{\alpha\gamma}(K_i)_\gamma^\beta p_\alpha p_\beta$, where $K_i$ are operator fields known as Killing $(1,1)$-tensors. If they algebraically commute, i.e., $K_iK_j=K_jK_i$, then the Haantjes torsion of $K_i$ vanishes.
In the case, when  $K_i$ have no Jordan blocks, the result is proved in \cite{KM1980}. The general case was solved only recently in \cite{BKM_generalseparation}.

In the theory of infinite-dimensional integrable systems, the condition of vanishing of the Haantjes torsion naturally appears in the study of integrable systems of hydrodynamic type, i.e., quasilinear evolutionary PDE systems
\begin{equation} \label{eq:1_m}
u_t= L(u)u_x.
\end{equation}
for an unknown vector function $u=(u^1(x,t), \dots, u^n(x,t))^\top$.
The subscripts $x$ and $t$ denote partial derivatives, and $L=L(u)$ is an $n\times n$ matrix viewed as an operator field written in local coordinates $u^1,\dots, u^n$. Such systems model many processes in mathematical physics and also appear in many purely  mathematical problems, see e.g. \cite{Serre3,tsarev1}.
 Integrability of \eqref{eq:1_m}  is equivalent to  the existence of $n$ mutually commuting symmetries which are linearly independent at every point (see e.g. \cite{tsarev, tsarev1, BKMRiemanninvariants2026}). In this context, the importance of the vanishing of the Haantjes torsion $\mathcal H_L$ for integrability was observed in many publications, e.g. \cite{FM2007, Sharipov, Magri2018}. It was conjectured in \cite{BKMRiemanninvariants2026} that vanishing of the Haantjes torsion is a necessary condition for integrability of \eqref{eq:1_m},  if $L$ is $\gl$-regular.  The conjecture is verified up to dimension $10$, and completely proved in the case, when $L$ is pointwise diagonalisable, \cite{BKMRiemanninvariants2026}.

The relations between Haantjes operators and finite- and infinite-dimensional integrable systems discussed above suggest that the local  classification of Haantjes operators may play an important role in their study.  Namely, one can start with a Haantjes operator as a primary object, and then consider the other objects  as unknown (e.g., a metric $g$ with integrable geodesic flow or commuting symmetries and conservation laws  of \eqref{eq:1_m}). Note that in these  studies  the operator fields are customarily  assumed to be $\gl$-regular. 
For a discussion and implementation of this approach in the theory of separation of variables, see \cite{TT22b, TT16, KKM2024}. We also refer to the recent papers by P.\,Lorenzoni et al. \cite{Lorenzoni2025, Lorenzoni2026} and the classical paper by S.\,Tsarev \cite{tsarev}  devoted to the integration of systems \eqref{eq:1_m}, where the authors use a natural ansatz for the operator field $L(u)$ (either diagonal or block diagonal with Toeplitz blocks) which already assumes that the Haantjes torsion of $L$ vanishes. Our local classification of Haantjes operators provides the geometric underpinning for this body of work and, in particular, shows that no significantly different cases are lost in this approach. 

Our results also demonstrate that at least in the $\gl$-regular case, the relationship between Haantjes and Nijenhuis geometries is very close: near generic points, each Haantjes operator $L$ is a polynomial of a suitable Nijenhuis operator $M$.   Thus, in terms of potential applications,  we have a natural way to extend results and methods developed in the framework of Nijenhuis geometry to the more general case of Haantjes operators.  Incidentally, this paper itself can serve as a good illustration of such an extension:  two steps of our construction (splitting theorem and description of Haantjes operators with complex eigenvalues) are largely based on the ideas developed in \cite{nij1} for Nijenhuis operators.

\subsection{Scheme of the proof}

Although the equivalent Theorems \ref{thm:A} and \ref{thm:B} are easy to formulate, their proof is quite nontrivial. We start, in Section~\ref{sect:3}, with basic properties of the Haantjes torsion, which will then be used in the proofs.

Our first result, which is actually the most challenging from a technical point of view, can be considered a crucial step in the proof of Theorems~\ref{thm:A} and~\ref{thm:B}.  In what follows,  we assume that our operator fields are defined on a manifold $\mathsf M$  which,  for the sake of simplicity, can be understood as an open subset of $\R^n$. 

\begin{Theorem}\label{thm:1}
Let $L$ be a Haantjes operator which is pointwise similar to a Jordan block with a real eigenvalue $\lambda(\mathsf p)$ depending on $\mathsf p\in\mathsf M$. Then, locally, by an appropriate coordinate transformation, $L$ can be reduced to the upper triangular Toeplitz form \eqref{eq1}. Conversely, if $L$ reduces to an upper triangular Toeplitz form, then $L$ is Haantjes.
\end{Theorem}

A similar result holds for Jordan blocks related to complex conjugate eigenvalues. To state it, we first recall a well-known fact from matrix algebra. Assume that $L$ is a real $2n \times 2n$ matrix with characteristic polynomial
$$
\chi_L(t)=\det (L - t \,\Id) =\bigl(t^2 - 2a t + (a^2+ b^2)\bigr)^n, \quad a,b\in \R,\quad b\ne 0.
$$
In other words, $L$ admits one pair of complex conjugate eigenvalues $a \pm \mathrm{i} b$ of multiplicity $n$ and no other eigenvalues. We also assume that
$$
\operatorname{rank}\Bigl(L^2 - 2a L + (a^2+b^2)\Id\Bigr) = 2(n-1),
$$
or, equivalently, that the minimal polynomial of $L$ coincides with $\chi_L(t)$. In this case, by a suitable similarity transformation, $L$ can be reduced to the canonical form  $J_n(a\pm\mathrm{i}\, b)$ given by \eqref{eq:CompJblock_m}, i.e., real Jordan block associated with the pair of complex conjugate eigenvalues $\lambda^{\pm}=a\pm\mathrm{i}\, b$.

For Haantjes operators $L$ of this algebraic type, we have the following complex analogue of Theorem~\ref{thm:1}.

\begin{Theorem}\label{thm:1b}
Assume that $L$ is a Haantjes operator which is similar to a real Jordan block related to a pair of complex conjugate eigenvalues $\lambda^{\pm}(\mathsf p) = a(\mathsf p) \pm \mathrm{i}\, b(\mathsf p)$, $b\ne 0$ at each point $\mathsf p \in \mathsf M^{2n}$.   Then there exists a complex structure $J$ on $\mathsf M$ and a local complex coordinate system $z_1=x_1+\mathrm{i}y_1,\dots, z_n=x_n+\mathrm{i} y_n$ in which $L$ is given by a Toeplitz $n\times n$ matrix \eqref{eq1}  with $g_j= g_j(\mathsf p)=u_j(\mathsf p) +\mathrm{i} v_j(\mathsf p)$ being a complex valued smooth function (not necessarily holomorphic).  Equivalently,  in the real coordinates $x_1, y_1, x_2, y_2,   \dots, x_n, y_n$,  the operator $L$ is given by $2n\times 2n$ real matrix 
    \begin{equation}\label{eq1b}
    \begin{pmatrix}
     G_n & G_{n - 1} & \!\!\!\dots & G_2 & \!\! G_1  \\
     0 &  G_n & \!\!\!G_{n-1} & & \!\! G_2 \\
   0  &  0 &\!\!\! \ddots  & \ddots   &\!\!  \vdots \\
    \vdots & \vdots  &\!\!\! \ddots & G_n   & \!\! G_{n - 1}\\
     0 & 0 & \!\!\! \dots  & 0 &\!\!  G_n
    \end{pmatrix}    
    \end{equation}
with $G_j$ being  $2\times 2$ block of the form $G_j = \begin{pmatrix} u_j  &  - v_j \\ v_j & u_j  \end{pmatrix} \simeq g_j=u_j + \mathrm{i} v_j$.  Conversely,  for all smooth functions $u_1,v_1, \dots, u_n, v_n$, the operator \eqref{eq1b} is Haantjes. 
\end{Theorem}

As compared with the proof of Theorem~\ref{thm:1} (Sections \ref{subsec:4.1}, \ref{subsec:4.2}), the proof of Theorem~\ref{thm:1b} contains one additional nontrivial step: we prove that a Haantjes operator with no real eigenvalues admits  a canonical complex structure and, moreover, can be written as a polynomial of a holomorphic Haantjes operator; see Appendix (Section \ref{appendix}). Modulo this step, the proof of Theorem~\ref{thm:1b}  (Section \ref{subsec:4.3}) follows the same lines as that of Theorem~\ref{thm:1}.

\begin{Remark}
\rm{
The `canonical' forms  \eqref{eq1} and \eqref{eq1b} in Theorems \ref{thm:1} and \ref{thm:1b} are not unique since there is a large family of coordinate transformations that change $g_i$'s in \eqref{eq1} but preserve the Toeplitz form of $L$. They are parametrized by $n - 1$ functions of two variables and one function of one variable. The explicit description of such transformations is equivalent to the description of Nijenhuis operators of the form \eqref{eq1} (see \cite{chernin} for details). 
}    
\end{Remark}

The next important step, {\it Splitting Theorem} for Haantjes operators,  proved in Section \ref{sect:5},  is related to Haantjes operators with an arbitrary spectrum. Moreover, the spectral structure (Segre characteristic) of $L$ may vary from point to point.

\begin{Theorem}\label{thm:2}
Let $L$ be a Haantjes operator. Assume that at a point $\mathsf p\in\mathsf M$, its characteristic polynomial factorises as $\chi_{L(\mathsf p)}(\lambda) = p_1(\lambda)p_2(\lambda)$, where $p_1(\lambda)$ and $p_2(\lambda)$ are coprime monic polynomials. Then, in a neighborhood of $\mathsf p$, there exists a coordinate system 
$$
\underbrace{u^1, \dots, u^{k}}_{u}, \underbrace{v^1, \dots, v^{s}}_{v},  \quad \text{where} \quad k = \operatorname{deg} p_1(\lambda),  s = \operatorname{deg} p_2(\lambda)
$$
such that
\begin{enumerate}
    \item[\rm(i)] The operator field $L$ has the form
    \begin{equation}\label{form2}
     L = \left( \begin{array}{cc}
          A (u, v) & 0_{k \times s}  \\
          0_{s \times k} & B(u, v)
     \end{array}\right)   
    \end{equation}
    \item[\rm(ii)] $A$ is Haantjes with respect to $u$, and $B$ is Haantjes with respect to $v$;
    \item[\rm(iii)] $A A_{v^j} - A_{v^j} A = 0$,
    where $A_{v^j}$ is the matrix  $\pd{A}{v^j}=\pd{}{v^j} \left( A^i_k\right)$;
    \item[\rm(iv)] $B B_{u^i} - B_{u^i} B = 0$, where $B_{u^i}$ is the matrix $\pd{B}{u^i}=\pd{}{u^i} \left( B^j_k\right)$.
\end{enumerate}
Moreover, every block diagonal operator \eqref{form2} satisfying {\rm (ii), (iii), (iv)} is Haantjes.   
\end{Theorem}

\begin{Remark}\label{r1}
\rm{  
The generalised Nijenhuis torsion of level $\mathrm m$, introduced in \cite{Kosmann, TT21}, is defined by the following recursive formulas:
\begin{equation}\label{highhaam}
\begin{aligned}
\mathcal T^{(1)}_L(\xi, \eta) = & \mathcal N_L(\xi, \eta), \\
\mathcal T^{(\mathrm m)}_L(\xi, \eta) = & L^2 \mathcal T^{(\mathrm m - 1)}_L(\xi, \eta) + \mathcal T^{(\mathrm m - 1)}_L(L\xi, L\eta) - \\
& - L\mathcal T^{(\mathrm m - 1)}_L(L\xi, \eta) - L\mathcal T^{(\mathrm m - 1)}_L(\xi, L\eta), \quad \mathrm m = 2, 3, \dots.   
\end{aligned}    
\end{equation}
For $\mathrm m = 2$, this formula coincides with the classical definition of Haantjes torsion $\mathcal H_L = \mathcal T^{(2)}_L$. It turns out that the statement of Theorem \ref{thm:2} still holds, if we assume that $L$ has zero Nijenhuis torsion of level $\mathrm m$, with conditions (iii) and (iv) replaced by 
$$
\mathrm {ad}_A^{\mathrm m - 1} A_{v^j} = 0 \quad \text{and} \quad \mathrm {ad}_B^{\mathrm m - 1} B_{u^i} = 0.
$$
Here the operator $\operatorname{ad}_L: \gl(n) \to \gl(n)$ is defined as $\operatorname{ad}_L M = LM - ML$. 

In particular, assuming that $m=1$ and setting $\mathrm {ad}_A^{0} = \Id$, we obtain the splitting theorem for Nijenhuis operators (see Theorem 3.1 in \cite{nij1}). 
}    
\end{Remark}

As a straightforward corollary of Theorem \ref{thm:2} and Remark \ref{r1}, we obtain the following block diagonalisation theorem for an {\it arbitrary} operator $L$ with $\mathcal T^{(\mathrm m)}_L=0$ in a neighbourhood of an {\it arbitrary} point $\mathsf p\in\mathsf M$.   

\begin{Corollary}\label{cor:2.1}
Let $\mathcal T^{(\mathrm m)}_L=0$, $m\ge 2$. Then in a neighbourhood of each point $\mathsf p\in\mathsf M$, there exists a local coordinate  
system
$$
\underbrace{u^1_1, \dots, u^{m_1}_1}_{u_1}, \underbrace{u^1_2, \dots, u^{m_2}_2}_{u_2}, \dots,  \underbrace{u^1_s, \dots, u^{m_s}_s}_{u_s}
$$
in which $L = \operatorname{diag}(L_1,\dots, L_s)$  with $m_i\times m_i$ blocks $L_i$, $i=1,\dots, s$, such that
\begin{enumerate}
    \item[\rm(1)]  at the point $\mathsf p$, each block  $L_i(\mathsf p)$ has either a single real eigenvalue or a pair of complex conjugate eigenvalues; 
    \item[\rm(2)] $L_i$ is a generalised Nijenhuis operator of level $\mathrm m$ with respect to $u_i$, and $\operatorname{ad}_{L_i}^{\mathrm m-1}  \frac{\partial L_i}{\partial  u_j^\alpha} =0$ for $j\ne i$, $\alpha=1,\dots, m_j$.
    \end{enumerate}
\end{Corollary}

\begin{Remark}{\rm  For Nijenhuis operators (i.e., for $\mathrm m=1$),  the result is much stronger (see Theorem 3.2 in \cite{nij1}): $L$ splits into {\it direct sum} of Nijenhuis operators $L_1(u_1), \dots, L_s(u_s)$, i.e. each block $L_i$ depends only on the variables $u_i=(u_i^1,\dots,u_i^{m_i})$ related to this block. 
For generalised Nijenhuis operators $L$ of level $\mathrm m\ge 2$, the  block diagonalisability was proved in \cite{TT22} under the additional condition that the eigenvalues of $L$ are all real and their multiplicities are locally constant, see  \cite[Theorem 40 and Proposition 43]{TT22}.  These results, however, do not explain the dependence of each block on the variables related to other blocks.  Corollary \ref{cor:2.1}  (see item (2))  fills this gap and also allows complex eigenvalues and collisions of eigenvalues within the blocks. In this sense, Corollary \ref{cor:2.1} serves as a universal tool in (generalised) Haantjes geometry that can be applied in the most general situation, for {\it any} operator and at {\it any} point.     
}\end{Remark}

Our main Theorems \ref{thm:A} and \ref{thm:B} follow more or less directly from the above mentioned preliminary results, we explain this in Section \ref{sect:6}. 

Finally, the Appendix establishes a natural link between real Haantjes operators with no real eigenvalues and Haantjes operators on complex manifolds.  Our main (and quite surprising) observation is that, under a mild additional assumption, a $\mathbb C$-linear operator field on a complex manifold $(\mathsf M, J)$ is Haantjes  if and only if  $\mathcal H_L (\partial_{z_\alpha}, \partial_{z_\beta}) = 0$ and $\frac{\partial L}{\partial \bar z_\alpha}$ commutes with $L$  (Theorem \ref{thm:7}).  Furthermore, this property implies that if such an operator $L$ is $\gl$-regular, then   $L=p(\widetilde L)$ where $\widetilde L$ is a holomorphic Haantjes operator (i.e., $\frac{\partial \widetilde L}{\partial \bar z_\alpha}=0$) and $p(\cdot)$ is a polynomial with smooth but not necessarily holomorphic coefficients (Corollary \ref{cor:7.5}).  Recall that  for Nijenhuis operators, the situation is simpler,  since  in the same setting, the vanishing of $\mathcal N_L$  implies that $L$ itself is holomorphic (Theorem 3.3 and Lemma 3.1 in \cite{nij1}).  To the best of our knowledge, in the context of Haantjes geometry, these questions have never been discussed in the literature before.

\section{Preliminaries: basic properties of Haantjes operators}\label{sect:3}

We start with some general facts on Haantjes torsion and Haantjes operators.  
Similar to the Nijenhuis torsion (see Definitions 2.1--2.5 in \cite{nij1}),  the Haantjes torsion can be defined in several equivalent ways.  

\begin{Proposition}\label{prop:3.1}  The  Haantjes torsion $\mathcal H_L$ can be defined
\begin{itemize}
\item As a vector-valued 2-form (standard definition):
$$
\mathcal H_L (\xi,\eta) = 
L^2\mathcal N_L(\xi, \eta ) + \mathcal N_L(L\xi,L\eta ) -  L\Bigl(\mathcal N_L(\xi,L\eta ) +\mathcal N_L(L\xi, \eta )\Bigr). 
$$
In more detail:
\begin{equation}\label{eq:def1}
\begin{aligned}
\mathcal H_L (\xi,\eta) &= 
  L^4[\xi,\eta] - 2L^3\Bigl([L\xi,\eta] + [\xi, L\eta]\Bigr) +  L^2 \Bigl(4 [L\xi, L\eta]  + [L^2\xi,\eta]  + [\xi, L^2\eta] \Bigr) - \\
  &- 2L\Bigl([L^2\xi,L\eta] + [L\xi, L^2\eta]\Bigr) + [L^2\xi, L^2\eta].
\end{aligned}
\end{equation}

\item As a map from `vectors' to `endomorphisms'

\begin{equation}
\label{eq:def2}
\mathcal H_L (\xi, \cdot): \quad  \xi \quad \mapsto \quad    L^2 [L, \mathcal L_\xi L]  - 2L [L,\mathcal L_{L\xi}L]  + [L, \mathcal L_{L^2\xi}L]. 
\end{equation}

\item As a map from `$1$-forms' to `$2$-forms'
$$
\mathcal H_L : \quad \alpha(\cdot) \quad \mapsto \quad  \beta (\cdot, \cdot),
$$
where
\begin{equation}
\label{eq:haant3}
\begin{aligned}
\beta (\cdot, \cdot) &= \ddd\alpha (L^2\cdot, L^2\cdot) - 2\Bigl(\ddd(L^*\alpha) (L^2 \cdot, L\cdot) + \ddd(L^*\alpha) (L \cdot, L^2 \cdot)\Bigr) +\\
&+ 4 \,\ddd ({L^*}^2\alpha) (L\cdot, L\cdot) +  \ddd ({L^*}^2\alpha)(L^2\cdot, \cdot) + \ddd ({L^*}^2\alpha)(\cdot, L^2 \cdot) - \\ 
&- 2 \Bigl(   \ddd({L^*}^3 \alpha)(L\cdot, \cdot) +  \ddd({L^*}^3 \alpha)(\cdot, L\cdot)   \Bigr) +\\
&+ \ddd ({L^*}^4\alpha) (\cdot, \cdot).
\end{aligned}
\end{equation}
\end{itemize}

\end{Proposition}

The verification of the above formulas for $\mathcal H_L$ is straightforward, and we omit it.

Assume that a Haantjes operator $L$ has a block triangular form $L=\begin{pmatrix} L_1 & C \\ 0 & L_2\end{pmatrix}$. This is the same as saying that we have an $L$-invariant foliation $\mathcal F$ and consider a coordinate system $(u,v)$ adapted to $\mathcal F$, so that $u_i$'s are coordinates on the leaves of $\mathcal F$, while $v_j$'s parametrise these leaves, i.e., are coordinates on the local quotient space $\mathsf M/\mathcal F$.   In this setting we have the following fact  (cf. Proposition 2.4 in \cite{nij1}).

\begin{Proposition}\label{prop:3.2}
\begin{enumerate}
\item[$(1)$] For each fixed $v$, the block $L_1(u,v)$  is a Haantjes operator on the $\mathcal F$-leaf corresponding to $v$.
\item[$(2)$]  If the block $L_2$ depends on $v$ only, then $L_2(v)$ is a Haantjes operator on the local quotient space $\mathsf M/\mathcal F$. 
\end{enumerate}
\end{Proposition}

The next  fact is a fundamental property of operators with vanishing Nijenhuis torsion of arbitrary level $\mathrm m\ge 2$.  For Haantjes operators ($\mathrm m=2$), it is due to O.\,Bogoyavlenskiy \cite{bg},  the general case was treated by D.R.\,Nozaleda et al. in \cite{RTT23}.

\begin{Proposition}[Theorem 3 in \cite{RTT23}]\label{prop:3.3}
Let $L$ be an operator field such that $\mathcal T^{(\mathrm m)}_L = 0$ for $\mathrm m \geq 2$. Then for any operator field $g(L) = g_1 L^{r - 1} + \dots + g_r \Id$ with arbitrary functional coefficients $g_i$, one has $\mathcal T^{(\mathrm m)}_{g(L)} = 0$.
\end{Proposition}

\begin{Proposition}\label{prop:3.4}
Let $L$ be an operator field such that $\mathcal T^{(\mathrm m)}_L = 0$ for $\mathrm m \geq 2$. Assume that the dimension of $\operatorname{Ker} L$ is locally constant and $\operatorname{Ker} L^2 = \operatorname{Ker} L$. Then the distribution $\mathcal D = \operatorname{Ker} L$ is integrable. \end{Proposition}

\begin{proof}
First, let us show that for $\xi, \eta \in \mathcal D = \operatorname{Ker} L$, we have
$\mathcal T^{(\mathrm m)}_L (\xi, \eta)  = L^{2 \mathrm m} [\xi, \eta]$.
We prove it by induction. For $\mathrm m = 1$, we have
$$
\mathcal N_L(\xi, \eta) = L^2[\xi, \eta] + [L\xi, L\eta] - L[L\xi, \eta] - L[\xi, L\eta] = L^2[\xi, \eta].
$$
Next, if  $\mathcal T^{(\mathrm k)}_L (\xi, \eta)  = L^{2 \mathrm k} [\xi, \eta]$, then 
$$
\begin{aligned}
\mathcal T^{(\mathrm k + 1)}_L (\xi, \eta) & = L^2 \mathcal T^{(\mathrm k)}_L (\xi, \eta) + \mathcal T^{(\mathrm k)}_L (L \xi, L\eta) - L\,\mathcal T^{(\mathrm k)}_L (L\xi, \eta) - \\
& - L\,\mathcal T^{(\mathrm k)}_L (\xi, L\eta) = L^2 \mathcal T^{(\mathrm k)}_L (\xi, \eta) = L^{2\mathrm k + 2}[\xi, \eta],
\end{aligned}
$$
as required.   Hence,  if $\xi, \eta \in \operatorname{Ker} L$, then $[\xi,\eta] \in \operatorname{Ker} L^{2\mathrm m} = \operatorname{Ker} L$, implying the integrability of $\mathcal D = \operatorname{Ker} L$.
\end{proof}


\section{Proof of Theorems \ref{thm:1} and \ref{thm:1b}}\label{sect:4}

The proof of Theorem \ref{thm:1} consists of two parts presented in Sections \ref{subsec:4.1} and \ref{subsec:4.2}.  

\subsection{Triangularisation Lemma}\label{subsec:4.1}

\begin{Theorem}[Triangularization Lemma]\label{thm:5} 
Let $L$ be a Haantjes operator which is pointwise similar to the nilpotent Jordan $n\times n$ block. Then

\begin{itemize}
\item[\rm (i)]  the distributions
$
\mathcal D_k := \operatorname{Ker} L^{k} = \operatorname{Image} L^{n-k}  
$
are all integrable;
\item[\rm (ii)] by an appropriate coordinate transformation, $L$ can be reduced to a strictly upper triangular form.
\end{itemize}

\end{Theorem}

\begin{proof}
The second statement (ii)  follows automatically from the integrability of the distributions $\mathcal D_k$.

The  integrability of $\mathcal D^k$ (item (i) of Theorem  \ref{thm:5}) will follow from the analysis of the linear equation systems on the partial derivatives $\frac{\partial L^i_j}{\partial x_k}$ given by the Haantjes relations.  

Consider a linearly independent collection of smooth vector fields $e_1, \dots, e_n$ (moving frame)  such that
$$
L e_i = e_{i-1}, \quad i=2,\dots,n, \quad \mbox{and} \quad L e_1 = 0.
$$ 
Equivalently, this means that the matrix of $L$ in the basis $e_1, \dots, e_n$ is the standard nilpotent Jordan block (with $1$'s  above the diagonal) and
$\mathcal D_{k} = \operatorname{Span} (e_1, \dots, e_k)$.  All the matrix computations below will be done in the basis $e_1, \dots, e_n$.  The integrability of the distributions $\mathcal D_{k}$ is equivalent to the fact that the Lie bracket of $[e_i, e_j]$ is a linear combination of $e_\alpha$'s with $\alpha \le \max\{i,j\}$.  Our goal is to verify this condition.

To that end, we introduce the matrices $B_1, B_2, \dots, B_n$ defined by
\begin{equation}
\label{eq:defB}    
(B_k)^i_j = c^i_{kj}\quad\mbox{for}\quad [e_k, e_j]= \sum c_{kj}^ie_i. 
\end{equation}

Formula  \eqref{eq:def2} for the Haantjes torsion imposes the following matrix relations on  $B_1, B_2, \dots, B_n$.

\begin{Lemma}\label{lem:4.1}
In terms of $B_1, B_2, \dots, B_n$, the Haantjes relations $\mathcal H_L\equiv 0$ take the following form
\begin{equation}
\label{eq:haanB}
L^2 [L,[B_i,L]] - 2 L[L, [B_{i-1}, L]] + [L, [B_{i-2}, L]] = 0
\end{equation}
where $B_k=0$ for $k\le 0$.  
\end{Lemma}

\begin{proof}
Straightforward verification based on formula \eqref{eq:def2} for $\mathcal H_L$. 
\end{proof}

Introducing the matrices $C_i =  [L,[B_i,L]]$, we  rewrite \eqref{eq:haanB} in the form
\begin{equation}
\label{eq:haanC}
\begin{aligned}
L^2  C_1 &= 0, \\
L^2  C_2 - 2 LC_1 &= 0, \\
L^2  C_3 - 2 LC_2 + C_1 &= 0, \\
L^2 C_4 - 2 LC_3 + C_2 &= 0, \\
& \dots \\
L^2 C_n - 2 LC_{n-1} + C_{n-2} &= 0. \\
\end{aligned}
\end{equation}

Recall that $L$ is the standard nilpotent Jordan block so that $L^2 C$ is obtained from $C$ by shifting the rows of $C$ up by two positions and replacing the two bottom rows with zeros. In particular,  the solution $C_1$ of the first relation from \eqref{eq:haanC} is the matrix whose first two rows are arbitrary, and all the others vanish.  Using this observation, we can easily resolve relations  \eqref{eq:haanC}  with respect to  $C_k$.   Each of these matrices will have arbitrary two top rows, while the other rows will be uniquely expressed in terms of the preceding matrices $C_1,\dots, C_{k-1}$. Moreover, the last (bottom) 
$n-(k+1)$  rows of  $C_k$  will be equal to zero.  

The general solution will be as follows.  The first two rows of  $C_k$ will be denoted by  $a_k$ and $b_k$.   Then the rows of  $C_k$ have the following structure:
\begin{equation}
\label{eq:matricesC}
C_1 = \begin{pmatrix}  a_1 \\ b_1 \\ 0 \\ 0 \\ \vdots \\ 0  \end{pmatrix}, \quad
C_2 = \begin{pmatrix}  a_2 \\ b_2 \\ 2b_1 \\ 0 \\ \vdots \\ 0  \end{pmatrix}, \quad
C_3 = \begin{pmatrix}  a_3 \\ b_3 \\ 2b_2 {-} a_1 \\ 3b_1 \\ \vdots \\ 0  \end{pmatrix}, \quad
C_4 = \begin{pmatrix}  a_4 \\ b_4 \\ 2b_3 {-} a_2 \\ 3b_2 {-}2a_1 \\ 4b_1 \\ 0  \end{pmatrix}, \quad
C_5 = \begin{pmatrix}  a_5 \\ b_5 \\ 2b_4 {-} a_3 \\ 3b_3 {-}2a_2 \\ 4b_2{-}3a_1 \\ 5b_1  \end{pmatrix}, \quad
\end{equation} 

Now we need to reconstruct $B_i$ from the matrix equation $C_i = [L, [B_i, L]]$.   In order for such an equation to be consistent, the matrix $C_i$ must satisfy some compatibility conditions.   Assume that all of our matrices $C_i$ satisfy them and find a {\it preliminary} form of the corresponding matrices $B_i$.

  Let us rewrite formula \eqref{eq:matricesC} in a slightly different way by denoting the elements of the row $b_1$ by $\beta_1,\beta_2, \dots, \beta_n$  (we will only need the fact that the last non-zero rows of the matrices  $C_i$ are proportional to  $b_1$):
{\small
$$
C_1 = \begin{pmatrix}
* & * & * & \dots & * \\
\beta_1 & \beta_2 & \beta _3 & \dots & \beta_n \\
0 & 0 & 0 & \dots & 0 \\
\vdots & \vdots & \vdots & \ddots & \vdots \\
 \end{pmatrix}, \quad
C_2 = \begin{pmatrix}
* & * & * & \dots & * \\
* & * & * & \dots & * \\
2\beta_1 & 2\beta_2 & 2\beta _3 & \dots & 2\beta_n \\
0 & 0 & 0 & \dots & 0 \\
\vdots & \vdots & \vdots & \ddots & \vdots \\
\end{pmatrix}, \quad
C_3 = \begin{pmatrix}
* & * & * & \dots & * \\
* & * & * & \dots & * \\
* & * & * & \dots & * \\
3\beta_1 & 3\beta_2 & 3\beta _3 & \dots & 3\beta_n \\
0 & 0 & 0 & \dots & 0 \\
\vdots & \vdots & \vdots & \ddots & \vdots \\
 \end{pmatrix}
$$
$$
C_4 = \begin{pmatrix}
* & * & * & \dots & * \\
* & * & * & \dots & * \\
* & * & * & \dots & * \\
* & * & * & \dots & * \\
4\beta_1 & 4\beta_2 & 4\beta _3 & \dots & 4\beta_n \\
0 & 0 & 0 & \dots & 0 \\
\vdots & \vdots & \vdots & \ddots & \vdots \\
 \end{pmatrix}, \quad \dots
$$ 
}
Then  by a straightforward computation,  one can check that the corresponding matrices $B_i$  have the following form  (we take into account the fact that the $k$-th column of  $B_k$ is zero since $[e_k, e_k]=0$).
We show the answer for  $6\times 6$ matrices to demonstrate the phenomenon (here the stars $*$ denote those elements which are not important for our purposes):
{\small
$$
B_1 =\begin{pmatrix}
 0 & * &* &* & *& *\\
0 & * & * & * & *&* \\
 0 & 4\lambda_1 &* &* &* &* \\
 0& 0 &3\lambda_1 & * &* &* \\
0& 0 & 0 &2\lambda_1 & *&* \\
0  & 0 & 0 & 0  &\lambda_1 & *\\
\end{pmatrix}, \ \lambda_1=\frac{\beta_3}{5}, \quad
B_2 =\begin{pmatrix}
 * & 0 & * & *&* &* \\
* & 0 & * & * &* &* \\
\! -2\beta_3 & 0  &* &* &* & *\\
 0 & 0 &3\lambda_2 & * &* &* \\
0 & 0 & 0 &2\lambda_2 & *&* \\
0   & 0& 0 & 0  &\lambda_2 & *\\
\end{pmatrix}, \ \lambda_2=\frac{2\beta_4}{4}, \quad
$$
$$
B_3 =\begin{pmatrix}
 * & * & 0& *& *&* \\
* & * & 0 & * &* &* \\
 * & * & 0 &* &* &* \\
 \!-3\beta_3& \!\!\!-3\beta_4 & 0 & * &* &* \\
0& 0 & 0 &2\lambda_3 & *&* \\
0  & 0 & 0 & 0  &\lambda_3 & *\\
\end{pmatrix}, \ \lambda_3=\frac{3\beta_5}{3}, \quad
B_4 =\begin{pmatrix}
 * & * & * & 0 & *&* \\
* & * & * & 0 &* &* \\
 * & *  &* & 0 & *& * \\
 * & * & * & 0 &* &* \\
\!-4\beta_3 & \!\!\!-4\beta_4 & \!\!\!-4\beta_5 & 0 & *&* \\
0   & 0& 0 & 0  &\lambda_4 & *\\
\end{pmatrix},\ \lambda_4=\frac{4\beta_6}{2}, 
$$  
$$
B_5 =\begin{pmatrix}
 * & * & * & * & 0 & * \\
* & * & * & * & 0 & *\\
 * & *  &* & * & 0 & *\\
 * & * & * & * & 0 & * \\
* & * & * & * & 0 &* \\
\!-5\beta_3   & \!\!\!-5\beta_4 & \!\!\!-5\beta_5 & \!\!\!-5\beta_6  & 0 & *\\
\end{pmatrix},
$$

}

Finally, we need to take into account the fact that  $c^i_{jk} = -c^i_{kj}$, see relations \eqref{eq:defB}  defining the matrices $B_1,\dots,B_n$.    In terms of these matrices,  this skew symmetry condition means that the $k$-th column of $B_j$ is equal to the $j$-th column of $B_k$ with minus sign.  By comparing the corresponding columns, we come to the following conclusion.

The second column of  $B_1$ and first column of $B_2$: \quad $2\beta_3 = 4\lambda_1 = \frac{4}{5}\beta_3 \quad  \Rightarrow \quad \lambda_1=\beta_3=0$.  

The third column of   $B_2$ and second column of  $B_3$: \quad  $3\beta_4 = 3\lambda_2 = \frac{3}{2}\beta_4  \quad  \Rightarrow \quad \lambda_2=\beta_4=0$.   

The fourth column of $B_3$ and third column of $B_4$: \quad  $4\beta_5 = 2\lambda_3 = 2\beta_5  \quad  \Rightarrow \quad \lambda_3=\beta_5=0$.   

The fifth column of $B_4$ and fourth column of $B_5$: \quad  $5\beta_6 = \lambda_4 = 2\beta_6  \quad  \Rightarrow \quad \lambda_4=\beta_6=0$, and so on.

Thus,  all the coefficients  $\beta_i$ (starting from $\beta_3$)  and all  $\lambda_i$ vanish, 
so that the matrices  $B_i$ we are interested in have the following form
{\small
$$
B_1 =\begin{pmatrix}
 0 & * &* &* & *& *\\
0 & * & * & * & *&* \\
 0 & 0 &* &* &* &* \\
 0& 0 & 0 & * &* &* \\
0& 0 & 0 & 0 & *&* \\
0  & 0 & 0 & 0  & 0 & *\\
\end{pmatrix},  \quad
B_2 =\begin{pmatrix}
 * & 0 & * & *&* &* \\
* & 0 & * & * &* &* \\
 0 & 0  &* &* &* & *\\
 0 & 0 &0 & * &* &* \\
0 & 0 & 0 &0 & *&* \\
0   & 0& 0 & 0  &0 & *\\
\end{pmatrix},\quad
B_3 =\begin{pmatrix}
 * & * & 0& *& *&* \\
* & * & 0 & * &* &* \\
 * & * & 0 &* &* &* \\
 0& 0 & 0 & * &* &* \\
0& 0 & 0 &0 & *&* \\
0  & 0 & 0 & 0  &0& *\\
\end{pmatrix},  \quad
$$
$$
B_4 =\begin{pmatrix}
 * & * & * & 0 & *&* \\
* & * & * & 0 &* &* \\
 * & *  &* & 0 & *& * \\
 * & * & * & 0 &* &* \\
0 & 0 & 0 & 0 & *&* \\
0   & 0& 0 & 0  & 0 & *\\
\end{pmatrix},\quad  
B_5 =\begin{pmatrix}
 * & * & * & * & 0 & * \\
* & * & * & * & 0 & *\\
 * & *  &* & * & 0 & *\\
 * & * & * & * & 0 & * \\
* & * & * & * & 0 &* \\
0   & 0 & 0 & 0  & 0 & *\\
\end{pmatrix},
$$
}

This structure of matrices $B_k$ does not depend on their size and exactly means that the bracket $[e_i, e_j]$ is a linear combination of $e_\alpha$'s with $\alpha\le \max (i,j)$, as required. This completes the proof of Theorem \ref{thm:5}.
\end{proof}

\subsection{Reduction of a triangular Haantjes operator to a Toeplitz  form}\label{subsec:4.2}

 Let $L$ be a Haantjes operator that is pointwise similar to a Jordan block with eigenvalue $\lambda(x)$.   Notice that $L -\lambda(x)\cdot \Id$ is still a Haantjes operator, which is now pointwise similar to a nilpotent Jordan block. Clearly,  it is sufficient to prove Theorem \ref{thm:1} for $L -\lambda(x)\cdot \Id$. Therefore,  without loss of generality we may assume that $L$ is similar to the nilpotent Jordan block and moreover,  in view of Theorem \ref{thm:5}, is given by a strictly upper triangular matrix in some local coordinates $x_1,\dots,x_n$.  From now on, we will be working only with those coordinate systems that are adapted to the flag of distributions $\mathcal D_k = \operatorname{Ker} L^k$,  see Theorem \ref{thm:5}.   In other words, 
all the subsequent coordinate transformations will be triangular, i.e., 
$$
\begin{aligned}
(x_1)_{\mathrm{new}} &= h_1(x_1,x_2, \dots, x_n),\\
 (x_2)_{\mathrm{new}}  &= h_2(x_2,\dots,x_n),\\
\vdots \\
(x_{n-1})_{\mathrm{new}}  &= h_{n-1}(x_{n-1}, x_n),\\
(x_n)_{\mathrm{new}} &=h_n(x_n)
\end{aligned}
$$

Assuming that $L$ is strictly upper triangular, we will reduce it to a Toeplitz form by using a step-by-step procedure based on the following key statement. 

\begin{Proposition}\label{prop:4.1}
Let $L$ be a Haantjes operator that is strictly upper triangular and of the form
$
L = \begin{pmatrix}
L_{k+1} & B \\
0  & J_s
\end{pmatrix}$, where $s=n-(k+1)$ and $J_s$ is the constant nilpotent Jordan $s\times s$ block.
Then by an appropriate coordinate transformation of the form $(x_{\mathrm{new}})_{k+1} = h(x_{k+1}, \dots, x_n)$ and suitable polynomial with functional coefficients
\begin{equation}
\label{eq:polyn}
p(t) = a_1 t + a_2 t^2 +\dots + a_{n-1} t^{n-1},\quad a_1\ne 0, 
\end{equation}
the operator $p(L)$ can be reduced to the form  
$$
p(L) = \begin{pmatrix}
\widetilde L_{k} & \widetilde B \\
0  & J_{s+1}
\end{pmatrix},
$$
where $J_{s+1}$ is the constant nilpotent Jordan $(s+1)\times (s+1)$ block.
\end{Proposition}

This proposition will immediately lead us to the proof of Theorem \ref{thm:1}.   Indeed, at the end of this procedure,  after $n-1$ steps, we construct  polynomials  $p_i(\cdot)$ and a new coordinate system such that $P(L) = p_{n-1}(p_{n-2} ( \dots (p_1(L))))=J_n$, where $J_n$ is the standard nilpotent Jordan $n\times n$ block.     Since $L$ and $J_n$ commute and $J_n$ is $\gl$-regular,  we can write $L$ as a polynomial of $J_n$, implying that $L$ is Toeplitz in the new coordinates, as required.   

\begin{proof}[Proof of Proposition \ref{prop:4.1}]
The first step of our procedure is the case, when
$$
L = \begin{pmatrix}
 0 & * & *  & *  & * \\
  & \ddots&\ddots  & & \\
  & & 0 & *  & * \\
  & & & 0 & b \\
  & & & & 0  \\
   \end{pmatrix}
$$
We need to make $b$ equal to $1$.  This can be done by replacing $L$ with $L_{\mathrm{new}} = p(L) = \frac{1}{b} L$.

Next,  consider an operator $L$ from Proposition \ref{prop:4.1}.  Let $k+1$ be the size of the $L_1$-block and consider the $(k+1)$-th row of the matrix $L$:
$$
\Bigr(\underbrace{0 \ \dots \ 0}_{k}\ 0 \ b_1 \ \dots \ b_{s-1} \ b_{s} \Bigr), \qquad  s=n-k-1. 
$$ 

\begin{Lemma}\label{lem:4.2}
The functions $b_1,\dots , b_{s-1}$   (i.e., all $b_i$'s except from the last one) do not depend on the variables $x_1,\dots, x_k$.
\end{Lemma}

\begin{proof}[Proof of Lemma \ref{lem:4.2}]
Consider the natural flag of integrable distributions $\mathcal D_1 \subset \mathcal D_2 \subset \dots$, where 
$\mathcal D_i = \operatorname{Span}(e_1, e_2,\dots, e_i)$, where $e_k=\partial_{x_k}$.  We have $L(\mathcal D_{i+1})= \mathcal D_{i}$, $L^2(\mathcal D_{i+2})=\mathcal D_i$,
and $[\mathcal D_i, \mathcal D_j] \subset \mathcal D_{\max\{i,j\}}$.

For convenience, we divide the  local coordinates into two groups and denote them by  $x_1,\dots, x_k$ and  $y_1=x_{k+1},\dots, y_{s+1}=x_n$.

Let $i\in \{1,2,\dots, k\}$.  Then there exists a vector field  
$$
\mbox{$\xi\in \mathcal D_{i+2}$ such that $L^2\xi = e_i = \partial_{x_i}$.}
$$  

We also take
$$
\eta = e_{k+2+j} = \partial_{y_{j+2}},\quad j=1,\dots,s-1.
$$

After this, we compute the  value of 
$$
\langle \dd y^1 , \mathcal H_L(\xi,\eta)\rangle.
$$

Since $L$ is Haantjes, this value equals zero, which gives us some relations for partial derivatives in question.
Notice first of all that
\begin{equation}
\label{eq:forD}
\begin{aligned}
L\eta = L(\partial_{y_{j+2}}) &= \partial_{y_{j+1}} + b_{j+1} \partial_{y_1}+\dots, \\
 L^2 \eta = L^2(\partial_{y_{j+2}}) &= \partial_{y_{j}} + b_{j} \partial_{y_1} +\dots  \quad \mbox{for $j>1$},\\
 L^2(\partial_{y_3}) &=  b_{1} \partial_{y_1} +\dots   \ \quad\quad\quad  \mbox{for $j=1$},
\end{aligned}
\end{equation}
where $\dots$ denote the terms from $\mathcal D_k$.
We will use the standard formula
$$
\begin{aligned}
\mathcal H_L (\xi,\eta) &= 
  L^4[\xi,\eta] - 2L^3\Bigl([L\xi,\eta] + [\xi, L\eta]\Bigr) +  L^2 \Bigl(4 [L\xi, L\eta]  + [L^2\xi,\eta]  + [\xi, L^2\eta] \Bigr) - \\
  &- 2L\Bigl([L^2\xi,L\eta] + [L\xi, L^2\eta]\Bigr) + [L^2\xi, L^2\eta].
\end{aligned}
$$

Substituting  $\xi$ and $\eta$ and taking into account relations \eqref{eq:forD}, we get for $j>1$:
$$
\begin{aligned}
 \mathcal H_L (\xi,\eta)  &= 
  L^4[\xi, \partial_{y_{j+2}}] - \\ & 2L^3\Bigl([L\xi,\partial_{y_{j+2}}] + [\xi, \partial_{y_{j+1}} + b_{j+1} \partial_{y_1}+\dots]\Bigr) +  \\ &
  L^2 \Bigl(4 [L\xi, \partial_{y_{j+1}} + b_{j+1} \partial_{y_1}+\dots]  + [L^2\xi,\partial_{y_{j+2}}]  + [\xi, \partial_{y_{j}} + b_{j} \partial_{y_1}+\dots] \Bigr) - \\
  & 2L\Bigl([L^2\xi,\partial_{y_{j+1}} + b_{j+1} \partial_{y_1}+\dots] + [L\xi, \partial_{y_{j}} + b_{j} \partial_{y_1}+\dots]\Bigr) + \\ & [L^2\xi, \partial_{y_{j}} + b_{j} \partial_{y_1}+\dots].
\end{aligned}
$$

Each Lie bracket of vector fields in this formula belongs to a certain distribution $\mathcal D_\alpha$.  In the formula below, we indicate the corresponding distribution for each of them:
$$
\begin{aligned}
 \mathcal H_L (\xi,\eta)  &= 
  L^4 (\mathcal D_{i+2}) - \\ & 2L^3\Bigl( \mathcal D_{i+1} + \mathcal D_{\max\{i+2, k+1\}}\Bigr) +  \\ &
  L^2 \Bigl(\mathcal D_{\max\{i+1, k+1\}}  + \mathcal D_i + \mathcal D_{\max\{i+2, k+1\}} \Bigr) - \\
  & 2L\Bigl(\mathcal D_{\max\{i, k+1\}} + \mathcal D_{\max\{i+1, k+1\}}\Bigr) + \\ & [\partial_{x_i}, \partial_{y_{j}} + b_{j} \partial_{y_1}+\dots].
\end{aligned}
$$

The first four terms belong to $\mathcal D_k$  (we use the fact that   $L^k (\mathcal D_m)\subset D_{m-k}$).  This implies that
$$  
\langle \dd y_1 ,  \mathcal H_L (\xi,\eta)  \rangle = \langle  \dd y_1, [\partial_{x_i}, \partial_{y_{j}} + b_{j} \partial_{y_1}+\dots]\rangle =
\langle  \dd y_1, [\partial_{x_i},  b_{j} \partial_{y_1}]\rangle = \frac{\partial b_j}{\partial x_i} = 0,
$$
where $i=1,\dots,k$ and $j=2,\dots, s-1$,
as required.    For $j=1$, i.e., $\eta = \partial_{y_3}$,  the verification is basically the same  (the small difference between $j=1$ and $j>1$ in \eqref{eq:forD} does not affect anything),  leading to the required conclusion that  $\frac{\partial b_1}{\partial x_i} = 0$.    \end{proof}

After this we continue working with the lower diagonal block of $L$ of the form
$$
\widetilde J_{s+1} = \begin{pmatrix}
0 & b_1 & \dots & b_{s-1} & b_s \\
 & 0 & 1  & & \\
 & & \ddots&\ddots  & \\
 & & & 0& 1 \\
 & & & &0 \\
 \end{pmatrix}
$$ 
In other words, $\widetilde J_{s+1} $ is obtained from the $J_s$-block of $L$ by extending it up to $(s+1)\times (s+1)$ block.

First of all, we make $b_s$ equal to zero by changing $\widetilde J_{s+1} $ with  $(\widetilde J_{s+1} )_{\mathrm{new}}=p(\widetilde J_{s+1} ) =\widetilde J_{s+1} -\frac{b_s}{b_1} \, \widetilde J_{s+1} ^{s}$. The other components of $\widetilde J_{s+1} $ do not change so that the new block \footnote{The same polynomial $p$ is used to replace the whole operator $L$ with $p(L)$, see the statement of Proposition \ref{prop:4.1}.} 
\begin{equation}
\label{eq:h1}
J = (\widetilde J_{s+1} )_{\mathrm{new}} =\begin{pmatrix}
0 & b_1 & \dots & b_{s-1} & 0 \\
 & 0 & 1  & & \\
 & & \ddots&\ddots  & \\
 & & & 0& 1 \\
 & & & &0 \\
 \end{pmatrix},  \quad b_i = b_i(y_1,\dots,y_{s+1}),  \ b_1\ne 0.
\end{equation}
Since this block does not depend on $x_1, \dots, x_k$, then by Proposition \ref{prop:3.2},  it is a Haantjes operator.

The next lemma gives necessary and sufficient conditions for an operator of form \eqref{eq:h1} to be Haantjes.

\begin{Lemma}\label{lem:4.3}
The operator \eqref{eq:h1} is Haantjes if and only if the differential form 
$$
\beta = \frac{1}{b_1} \,\dd y_1 - \frac{b_2}{b_1}\, \dd y_2 - \dots -  \frac{b_{s-1}}{b_1}\, \dd y_{s-1}
$$
is closed  (here we think of $\beta$ as a differential form in variables $y_1,\dots,y_{s-1}$ while $y_s$ and $y_{s+1}$ are treated as parameters).
\end{Lemma}
\begin{proof} Denote $e_i = \partial_{y_i}$. One can easily check that $\mathcal N_J (\xi, \eta) = B(\xi, \eta) e_1$.  
Since $Je_1=0$, the formula for the Haantjes torsion simplifies and takes the form 
$$
\mathcal H_J (\xi,\eta) = 
J^2\mathcal N_J(\xi; \eta ) + \mathcal N_J(J\xi;J\eta ) -  J\Bigl(\mathcal N_J(\xi;J\eta ) +\mathcal N_J(J\xi; \eta )\Bigr)=\mathcal N_J(J\xi;J\eta ).
$$ 
Let us compute  $\mathcal H_J (e_i, e_j)  = \mathcal N_J(Je_i;Je_j )$ assuming that $i,j>1$   (if $i$ or $j$ equals 1, we get zero automatically, since $Je_1=0$).
$$
\mathcal N_J(Je_i;Je_j ) =  J^2[Je_i;Je_j] - J [J^2e_i;Je_j] - J [Je_i;J^2e_j] + [J^2 e_i, J^2 e_j].
$$
It can be easily seen that only the last term in this formula is not identically zero.  This implies, in particular, that $\mathcal H_J (e_i, e_j) =0$ if either $i\le 2$ or $j\le 2$.  Thus, for $i,j \ge 4$ we get:
\begin{equation}
\label{eq:72}
\begin{aligned}
\mathcal H_J(e_i;e_j ) &=[J^2 e_i, J^2 e_j]  = [e_{i-2} + b_{i-2}e_1, e_{j-2} + b_{j-2}e_1] = \\
&= \left(\frac{\partial b_{j-2}}{\partial y_{i-2}} + b_{i-2} \frac{\partial b_{j-2}}{\partial y_{1}} - 
\frac{\partial b_{i-2}}{\partial y_{j-2}} - b_{j-2} \frac{\partial b_{i-2}}{\partial y_{1}} \right) e_1.
\end{aligned}
\end{equation}
Similarly, for $i=3$ and $j\ge 4$:
\begin{equation}
\label{eq:72b}
\begin{aligned}
\mathcal H_J(e_3;e_j ) &=[J^2 e_3, J^2 e_j]  = [b_1 e_1, e_{j-2} + b_{j-2}e_1] = \\
&= \left(  b_{1} \frac{\partial b_{j-2}}{\partial y_{1}} - 
\frac{\partial b_{1}}{\partial y_{j-2}} - b_{j-2} \frac{\partial b_{1}}{\partial y_{1}} \right) e_1.
\end{aligned}
\end{equation}

Denoting $\alpha=i-2$ and $\beta = j-2$,  we conclude that $J$ is Haantjes if and only if the following relations hold
\begin{equation}
\label{eq:73}
\begin{aligned}
\frac{\partial b_{\beta}}{\partial y_{\alpha}} + b_{\alpha} \frac{\partial b_{\beta}}{\partial y_{1}} - 
\frac{\partial b_{\alpha}}{\partial y_{\beta}} - b_{\beta} \frac{\partial b_{\alpha}}{\partial y_{1}} &= 0,   \qquad \mbox{for all $\alpha,\beta = 2,\dots, s-1$}.\\
 b_{1} \frac{\partial b_{\beta}}{\partial y_{1}} - 
\frac{\partial b_{1}}{\partial y_{\beta}} - b_{\beta} \frac{\partial b_{1}}{\partial y_{1}} &= 0,   \qquad \mbox{for $\alpha=1$, $\beta = 2,\dots, s-1$}.
\end{aligned}
\end{equation}

It is straightforward to see that these relations are equivalent to the closedness of $\beta$  (as a $1$-form in variables $y_1,\dots,y_{s-1}$).  
\end{proof}

Since $\beta= \frac{1}{b_1} \,\dd y_1 - \frac{b_2}{b_1}\, \dd y_2 - \dots -  \frac{b_{s-1}}{b_1}\, \dd y_{s-1}$ is closed, we can locally find $(y_1)_{\mathrm{new}}=F(y_1,\dots, y_{s+1})$ such that 
$$
\dd (y_1)_{\mathrm{new}} = \beta + F_{y_s}\dd y_s + F_{y_{s+1}}\dd y_{s+1}.
$$   
It is easy to see that in the new variables
$(y_1)_{\mathrm{new}}, y_2, \dots, y_{s+1}$, the operator $J$ takes the form 
$$
J_{\mathrm{new}} = \begin{pmatrix}
0 & 1 & \dots & 0 & c(y) \\
 & 0 & 1  & & 0\\
 & & \ddots&\ddots  &\vdots \\
 & & & 0& 1 \\
 & & & &0 \\
 \end{pmatrix}
$$

Finally, to kill $c(y)$, we replace $J$ with $J_{\mathrm{new}} = J -c(y)\,J^s$.  (As above,   we simultaneously replace the whole operator $L$ with $L - c(y)\,L^s$).

As a result, $J$ is reduced to the standard nilpotent block $J_{s+1}$ as required.  This completes the proof of 
Proposition \ref{prop:4.1} and therefore the Toeplitz normal form theorem (Theorem \ref{thm:1}). \end{proof}

\subsection{Toeplitz normal form for a complex Jordan block}\label{subsec:4.3}

Here we prove Theorem \ref{thm:1b} as a corollary of (the proof of) Theorem \ref{thm:1} and properties of complex Haantjes operators proved in the Appendix.  
Let  $L$ be a real matrix which is similar to a real Jordan block related to a pair of complex conjugate eigenvalues $\lambda^{\pm} = a \pm \mathrm{i} \, b$. 
Then by Corollary \ref{cor:7.5},  we introduce a complex structure $J$ such that $LJ = JL$ so that in complex coordinates $z_1,\dots, z_n$,  the operator $L$ is given by an $n\times n$ complex matrix $L^{\mathbb C}$. Moreover,  
$$
L^{\mathbb C} = p(\widetilde L),
$$
where $\widetilde L$ is a holomorphic Haantjes operator and $p(\cdot)$ is a polynomial whose coefficients are complex valued smooth  functions in $z_1,\dots,z_n$ (not necessarily holomorphic).    Clearly $\widetilde L$ is similar to a Jordan block. 

We can continue working with the holomorphic Haantjes operator $\widetilde L(z)$, $z=(z_1,\dots, z_n)$  in the same way as with smooth Haantjes operators $L(x)$, $x=(x_1,\dots, x_n)$ similar to a Jordan block.  In other words, the above reduction procedure explained in detail in Sections \ref{subsec:4.1} and \ref{subsec:4.2}  works without any changes in the holomorphic setting.    Applying this procedure to $\widetilde L(z)$, we find a {\it holomorphic} coordinate transformation $(z_1,\dots, z_n) \mapsto (z_1',\dots,z_n')$ which brings $\widetilde L$ to a complex upper triangular Toeplitz form.  It remains to notice that a polynomial of an upper triangular Toeplitz matrix is also an upper triangular Toeplitz matrix, which completes the proof of Theorem \ref{thm:1b}.

\section{Proof of Theorem \ref{thm:2} and Remark \ref{r1}}\label{sect:5}

We start with an interesting algebraic fact. Let $\mathbb K= \R$ or $\mathbb C$. Consider the non-linear map 
$\Phi: {\mathbb K}^k \times {\mathbb K}^s \to {\mathbb K}^{k + s}$ defined 
as follows. For $a = (a_1, \dots, a_k)\in {\mathbb K}^k$ and $b=(b_1, \dots, b_s)\in {\mathbb K}^s$,  we set $\Phi(a, b) = c$, where $c=(c_1, \dots, c_{k+s}) \in {\mathbb K}^{k+s}$ is defined from the polynomial relation
    \begin{equation}
    \label{l1_1}
    (\lambda^k + a_1 \lambda^{k - 1} + \dots + a_k)(\lambda^s + b_1 \lambda^{s - 1} + \dots + b_s)=
    \lambda^{k+s} + c_1\lambda^{k+s-1} + \dots + c_{k+s}.
    \end{equation}

\begin{Lemma}\label{lem:5.1}
Let  the polynomials $p_1(\lambda)=\lambda^k + a_1 \lambda^{k - 1} + \dots + a_k$ and $p_2(\lambda)=\lambda^s + b_1 \lambda^{s - 1} + \dots + b_s$ be coprime and $\Phi(a,b)=c\in {\mathbb K}^{k + s}$. Then locally,  in a neighborhood of $(a,b)\in {\mathbb K}^k\times {\mathbb K}^s$,  the mapping $\Phi$ is invertible, that is, there exist neighborhoods $U_{(a,b)} \subset {\mathbb K}^k\times {\mathbb K}^s$ and $U_c\subset {\mathbb K}^{k+s}$  such that $\Phi: U_{(a,b)} \to U_c$ is bijective and $\Phi^{-1}: U_c \to U_{(a,b)}$ is analytic.
\end{Lemma}
\begin{proof}
W.l.o.g we assume $k \geq s$. Expanding the left hand side of \eqref{l1_1}, we obtain the explicit formulas for $c=\Phi(a,b)$: 
\begin{equation}\label{vans}
\begin{aligned}
c_1 & = a_1 + b_1, \\
c_2 & = a_2 + a_1 b_1 + b_2, \\
& \dots \\
c_s & = a_s + a_{s - 1} b_1 + \dots + b_s \\
c_{s + 1} & = a_{s + 1} + a_s b_1 + \dots + a_1 b_s, \\
& \dots \\
c_{k} & = a_k + a_{k - 1} b_1 + \dots + a_{k - s} b_s, \\
c_{k + 1} & = a_k b_1 + a_{k - 1} b_2 + \dots + a_{k - s + 1} b_s, \\
& \dots \\
c_{k + s} & = a_k b_s.
\end{aligned}
\end{equation}
Explicit calculation of the Jacobi matrix  (differential) of $\Phi$ yields the matrix
$$
\begin{pmatrix}
1 & 0 & \dots & 0 & 1 & 0 & \dots & 0 \\
b_1 & 1 & \dots & 0 & a_1 & 1 & \dots & 0 \\
b_2 & b_1 & \ddots & 0 & a_2 & a_1 & \ddots & 0 \\
\vdots & \vdots & \ddots & 1 & \vdots & \vdots & \ddots & 1 \\
b_s & \!\!\! b_{s - 1} &  \dots & \vdots & a_k & \!\!\! a_{k - 1} & \dots & \vdots \\
0 & b_s & \ddots & \vdots & 0 & a_k & \ddots & \vdots \\
\vdots & \vdots & \ddots & \!\!\! b_{s - 1} & \vdots & \vdots &  \ddots & \!\!\! a_{k - 1} \\
0 & 0 & \dots & b_s & 0 & 0 & \dots & a_k \\ 
\end{pmatrix}.
$$
known (up to a sign) as the Sylvester matrix for the polynomials $p_1(\lambda)$ and $p_2(\lambda)$. Up to a sign, its determinant coincides with the resultant of the polynomials $p_1(\lambda), p_2(\lambda)$ (Theorem 3.1 in \cite{prasolov}). As the polynomials are coprime, i.e., have no common roots,  the resultant is not zero. Hence, $\det \Bigl(\dd \Phi(a,b)\Bigr)\ne 0$ and the statement follows from the inverse function theorem in the analytic category.
\end{proof}

We will use the following corollary of Lemma \ref{lem:5.1}. 

\begin{Corollary}\label{cor:5.1}
Let $\chi(\lambda)$ be a monic polynomial of degree $k + s$ with coefficients that are real smooth (analytic) functions on a manifold $\mathsf M$. Assume that at a given point $\mathsf p\in\mathsf M$, the polynomial $\chi(\lambda)$ can be factored into two coprime monic polynomials $p_1(\lambda)$ and $p_2(\lambda)$ of degrees $k$ and $s$, respectively. Then locally, in some neighborhood $U(\mathsf p)$, there exist unique monic polynomials $\chi_1(\lambda)$ and $\chi_2(\lambda)$ of degrees $k$ and $s$, respectively, such that

1.   $\chi(\lambda) = \chi_1(\lambda) \chi_2(\lambda)$ for all points of $U(\mathsf p)$;

2.   their coefficients are smooth (analytic) functions on $U(\mathsf p)$;

3.  at the point $\mathsf p$, the polynomials $\chi_1(\lambda)$ and $\chi_2(\lambda)$ coincide with $p_1(\lambda)$ and $p_2(\lambda)$.

In other words, the coprime factorisation at $\mathsf p\in\mathsf M$ can be uniquely extended to the entire neighborhood $U(\mathsf p)$.
\end{Corollary}

Now we are ready to prove Theorem \ref{thm:2} in the (more general) version of Remark \ref{r1}. 

Let $L$ be an operator satisfying $\mathcal T^{(\mathrm m)}_L = 0$, i.e., the Nijenhuis $\mathrm m$-level torsion of $L$ vanishes. Assume that at a point $\mathsf p\in\mathsf M$,  the characteristic polynomial of $L$ admits a factorisation into two coprime monic polynomials $\chi_{L(\mathsf p)}(\lambda)=p_1(\lambda) p_2(\lambda)$ and apply  Corollary \ref{cor:5.1}. We obtain two coprime polynomials $\chi_1(\lambda), \chi_2(\lambda)$ with smooth coefficients, such that $\chi_1(\lambda) \chi_2(\lambda) = \chi_L(\lambda)$. 

We define a pair of operators $M_1 = \chi_1(L)$ and $M_2 = \chi_2(L)$. If $\mathcal T^{(\mathrm m)}_L = 0$, then by Proposition \ref{prop:3.3},  $\mathcal T^{(\mathrm m)}_{M_1} = \mathcal T^{(\mathrm m)}_{M_2} = 0$. Due to the choice of the polynomials $\chi_1$ and $\chi_2$,  both $M_1$ and $M_2$ satisfy the conditions of Proposition \ref{prop:3.4}. Moreover,  $T \mathsf M = \operatorname{Ker} M_1 \oplus \operatorname{Ker} M_2$.  Thus, by Proposition \ref{prop:3.4}, we obtain two complementary integrable distributions $\mathcal D_1=\operatorname{Ker} M_1$ and $\mathcal D_2=\operatorname{Ker} M_2$,  both invariant with respect to $L$. 

Hence, there exist local coordinates $u, v$  such that $\mathcal D_1 = \operatorname{Span}(\partial_{u^1},\dots,\partial_{u^k})$ and  $\mathcal D_2 = \operatorname{Span}(\partial_{v^1},\dots,\partial_{v^s})$.  They automatically satisfy the property (i) from Theorem \ref{thm:2}:  in this  coordinate system,  $L = \begin{pmatrix} A(u,v) & 0 \\ 0 & B(u,v)\end{pmatrix}$.  

The vanishing of the torsion $\mathcal T^{(\mathrm m)}$  for $A$ and $B$ with respect to $u$ and $v$ is obvious (property (ii) in Theorem \ref{thm:2}).  To verify (iii) and (iv), we need the following lemma.

\begin{Lemma}\label{lem:5.2}
Let $L$ be in the form \eqref{form2}. Then the following formula holds:
\begin{equation}\label{dark}
\begin{aligned}
\mathcal T^{(\mathrm m)}_L (\partial_{u^i}, \partial_{v^j}) & = (-1)^{{\mathrm m}-1}\sum_{\mathrm k = 0}^{\mathrm m} (-1)^{\mathrm k} \binom{\mathrm m}{\mathrm k}(A^{\mathrm m - \mathrm k})^q_i (\operatorname{ad}^{\mathrm m - 1}_B B_{u^q})^s_j (B^{\mathrm k})^\beta_s \partial_{v^\beta} - \\
& - (-1)^{{\mathrm m}-1}\sum_{\mathrm k = 0}^{\mathrm m} (-1)^{\mathrm k} \binom{\mathrm m}{\mathrm k}(B^{\mathrm m - \mathrm k})^r_j (\operatorname{ad}^{\mathrm m - 1}_A A_{v^r})^s_i (A^{\mathrm k})^\alpha_s \partial_{u^\alpha}
\end{aligned}
\end{equation}
\end{Lemma}
\begin{proof}
We have
\begin{equation}\label{d4}
\begin{aligned}
\mathcal N_L (\partial_{u^i}, \partial_{v^j}) & = [L\partial_{u^i}, L\partial_{v^j}] - L[L\partial_{u^i}, \partial_{v^j}] - L[\partial_{u^i}, L\partial_{v^j}] = \\
& = [A^\alpha_i \partial_{u^\alpha}, B^\beta_j \partial_{v^\beta}] - L[A_i^\alpha \partial_{u^\alpha}, \partial_{v^j}] - L[\partial_{u^i}, B_j^\beta \partial_{v^\beta}] = \\ 
& = A^q_i \pd{B_j^\beta}{u^q} \partial_{v^\beta} - \pd{A^\alpha_i}{v^q} B^q_j \partial_{u^\alpha} + \pd{A_i^q}{v^j} A_q^\alpha \partial_{u^\alpha} - \pd{B_j^q}{u^i} B_q^\beta \partial_{v^\beta} = \\
& = \Big( A^q_i \pd{B_j^\beta}{u^q} - \pd{B_j^q}{u^i} B_q^\beta \Big)\partial_{v^\beta} - \Big( \pd{A^\alpha_i}{v^q} B^q_j - \pd{A_i^q}{v^j} A_q^\alpha\Big) \partial_{u^\alpha} = \\
& = \Big(A^q_i 
(\mathrm {ad}_B^0 B_{u^q})^\beta_{j} - (\mathrm {ad}_B^0 B_{u^i})^q_{j} B_q^\beta \Big) \partial_{v^\beta} + \dots.
\end{aligned}    
\end{equation}
Here and further, the dots stand for the terms that contain $\partial_{u^\alpha}$. We also use $\mathrm {ad}_B^0 = \Id$.

We will prove the formula for terms that only contain $\partial_{v^\beta}$. For $\partial_{u^\alpha}$, the proof is similar. We prove it by induction. The formula \eqref{d4} is the base of our induction. Now assume that the formula holds for $\mathrm m$, that is
$$
\mathcal T^{(\mathrm m)}_L (\partial_{u^i}, \partial_{v^j}) = (-1)^{{\mathrm m}-1}\sum_{\mathrm k = 0}^{\mathrm m} (-1)^{\mathrm k} \binom{\mathrm m}{\mathrm k}(A^{\mathrm m - \mathrm k})^q_i (\mathrm {ad}_B^{\mathrm m - 1} B_{u^q})^s_{j} (B^{\mathrm k})^\beta_s \partial_{v^\beta} + \dots
$$
By construction, we have
$$
\begin{aligned}
  \mathcal T^{(\mathrm m + 1)}_L  (\partial_{u^i}, \partial_{v^j}) & = L^2 \mathcal T^{(\mathrm m)}_L (\partial_{u^i}, \partial_{v^j}) - L \mathcal T^{(\mathrm m)}_L (\partial_{u^i}, L \partial_{v^j}) +  \mathcal T^{(\mathrm m)}_L (L \partial_{u^i}, L \partial_{v^j}) - L \mathcal T^{(\mathrm m)}_L (L \partial_{u^i}, \partial_{v^j})  = \\
 & =  (-1)^{{\mathrm m}-1} \sum_{\mathrm k = 0}^{\mathrm m} (-1)^{\mathrm k} \binom{\mathrm m}{\mathrm k}(A^{\mathrm m - \mathrm k})^q_i (\operatorname{ad}^{\mathrm m - 1}_B B_{u^q})^s_{j} (B^{\mathrm k + 2})^\beta_s \partial_{v^\beta} - \\ 
 & - (-1)^{{\mathrm m}-1} \sum_{\mathrm k = 0}^{\mathrm m} (-1)^{\mathrm k} \binom{\mathrm m}{\mathrm k}(A^{\mathrm m - \mathrm k})^q_i B^l_j (\operatorname{ad}_B^{\mathrm m - 1} B_{u^q})^s_{l} (B^{\mathrm k + 1})^\beta_s \partial_{v^\beta} + \\
 & + (-1)^{{\mathrm m}-1} \sum_{\mathrm k = 0}^{\mathrm m} (-1)^{\mathrm k} \binom{\mathrm m}{\mathrm k} (A^{\mathrm m - \mathrm k + 1})^q_i B^r_j (\operatorname{ad}^{\mathrm m - 1}_B B_{u^q})^s_{r} (B^{\mathrm k})^\beta_s \partial_{v^\beta} - \\ 
 & - (-1)^{{\mathrm m}-1} \sum_{\mathrm k = 0}^{\mathrm m} (-1)^{\mathrm k} \binom{\mathrm m}{\mathrm k} (A^{\mathrm m - \mathrm k + 1})^q_i (\operatorname{ad}^{\mathrm m - 1}_B B_{u^q})^s_{j} (B^{\mathrm k + 1})^\beta_s \partial_{v^\beta}  + \dots  = 
 \end{aligned}
 $$
 $$
 \begin{aligned}
 & = - (-1)^{{\mathrm m}-1} \sum_{\mathrm k = 0}^{\mathrm m} (-1)^{\mathrm k} \binom{\mathrm m}{\mathrm k} (A^{\mathrm m - \mathrm k + 1})^q_i (\operatorname{ad}^{\mathrm m}_B B_{u^q})^s_{j} (B^{\mathrm k})^\beta_s \partial_{v^\beta} \\ & + (-1)^{{\mathrm m}-1} \sum_{\mathrm k = 0}^{\mathrm m} (-1)^{\mathrm k} \binom{\mathrm m}{\mathrm k} (A^{\mathrm m - \mathrm k})^q_i (\operatorname{ad}^{\mathrm m}_B B_{u^q})^s_{j} (B^{\mathrm k + 1})^\beta_s \partial_{v^\beta} + \dots = \\
 &  =  (-1)^{\mathrm m}\sum_{k = 0}^{\mathrm m + 1} (-1)^{\mathrm k} \binom{\mathrm m +1}{\mathrm k}(A^{\mathrm m + 1 - \mathrm k})^q_i (\operatorname{ad}^{\mathrm m}_B B_{u^q})^s_{j} (B^{\mathrm k})^\beta_s \partial_{v^\beta} + \dots.
\end{aligned}
$$
Thus, the formula is proved, as well as the lemma.
\end{proof}

Without loss of generality we may assume that $L$ is invertible. If it is not, then we can replace $L$ by $L+ c\,\Id$ for some constant $c$,  and the new operator will still have vanishing Nijenhuis torsion of level $\mathrm m$ due to Proposition \ref{prop:3.3}. Since the blocks $A$ and $B$ have no common eigenvalues, we know that at each point, there exists a polynomial $g(t) = g_1 t^{r - 1} + \dots + g_n$ such that
$$
g(L) = g_1 L^{r - 1} + \dots + g_n \Id = \left( \begin{array}{cc}
          A (u, v) & 0_{k \times s}  \\
          0_{s \times k} & 0_{s \times s}
     \end{array}\right).
$$
The coefficients of such a polynomial are locally smooth functions. By Proposition \ref{prop:3.3}, $\mathcal T^{(\mathrm m)}_{g(L)} = 0$. Applying Lemma \ref{lem:5.2} we get
$$
(\operatorname{ad}^{\mathrm m - 1}_A A_{v^r})^s_i (A^{\mathrm m})^\alpha_s \partial_{u^\alpha} = 0.
$$
As $A$ is non-degenerate, we conclude that $\operatorname{ad}^{\mathrm m - 1}_A A_{v^r} = 0$ for all $1 \leq r \leq s$. In a similar fashion, we get $\operatorname{ad}^{\mathrm m - 1}_B B_{u^i} = 0$ for all $1 \leq i \leq k$.   This proves items  (iii) and (iv) of Theorem  \ref{thm:2}.  The last statement of Theorem  \ref{thm:2} follows from Lemma \ref{lem:5.2}.


\section{Proof of Theorems \ref{thm:A} and \ref{thm:B}}\label{sect:6}

The proof is based on the following lemma

\begin{Lemma}\label{lem:6.1}
Let $L(t)$ be a smooth family of $\gl$-regular matrices such that $L$ and $L'_t$ commute at each point. Then $L(t)$ takes the following form
$$
L(t) = a_0(t) \Id + a_1(t) L(0) + a_2(t) L^2(0)  + \dots + a_{n-1}(t) L^{n-1}(0). 
$$
In particular, if $L(0)$ is an upper triangular Toeplitz matrix \eqref{eq1} (or its complex analog \eqref{eq1b}), then $L(t)$ is a matrix of the same type for any $t$.
\end{Lemma}

Next we can proceed as follows.

Let  $L$ be a $\gl$-regular Haantjes operator and $\mathsf p \in \mathsf M$ be an algebraically generic point for $L$. Then according to Theorem \ref{thm:2}  (splitting theorem for Haantjes operators) and Corollary \ref{cor:2.1}, in a neighbourhood of $\mathsf p$ there exists a local coordinate system 
$$
\underbrace{x_1^1,\dots, x_1^{m_1}}_{x_1},  \underbrace{ x_2^1,\dots, x_2^{m_2}}_{x_2}, \dots , \underbrace{x_s^1,\dots, x_s^{m_s}}_{x_s}
$$
in which $L$ takes a block diagonal form 
$
L = \operatorname{diag} (L_1, \dots, L_s),
$
where each block $L_i$ is similar to either a Jordan block with a real eigenvalue or a real Jordan block related to a pair of complex conjugate eigenvalues.

Without loss of generality we assume that $\mathsf p$ is the origin of this coordinate system, i.e.,  the coordinates of this point are all zeros. According to Theorems \ref{thm:1} and \ref{thm:1b},  for each block $L_i$ the coordinates $x_i=(x_i^1,\dots, x_i^{m_i})$ can be chosen in such a way that at each point with coordinates $(0, \dots, 0, x_i, 0,\dots, 0)$, the block $L_i$ is an upper triangular Toeplitz matrix \eqref{eq1} or its complex analog \eqref{eq1b}.   

However, according to the splitting theorem (Theorem \ref{thm:2}),  the dependence of $L_i$ on the other group of variables $x_j=(x_j^1,\dots, x_j^{m_j})$, $j\ne i$,  is such that 
$L_i$ and $\frac{\partial L_i}{\partial x_j}$ commute.   Therefore, by Lemma \ref{lem:6.1},  $L_i$ is an upper triangular Toeplitz matrix  (or, respectively, its complex analog) at each point $(x_1,\dots, x_s)$.

Thus we have constructed a local coordinate system in which $L$ has a block-diagonal form with all the blocks $L_i$ being upper triangular Toeplitz matrices (Theorem \ref{thm:B}).  Now, in this coordinate system, consider the constant Jordan matrix
$M = \operatorname{diag}(M_1,\dots, M_s)$, where $M_i$, $i=1,\dots, s$,  is the standard (constant) Jordan block of the same size as $L_i$  with a real eigenvalue $\lambda_i$  or a pair of complex conjugate eigenvalues $\lambda_i^{\pm} = a_i \pm \mathrm{i} b_i$, $b_i\ne 0$ depending on the type of $L_i$.  We require that the eigenvalues from different blocks be distinct, so that $M$ is $\gl$-regular.    It remains to notice that upper triangular Toeplitz matrix commutes with the standard Jordan block, so that $L$ belongs to the centraliser of $M$ and therefore can be written as $L=p(M)$, where $p(\cdot)$ is a polynomial of degree $n-1$ with coefficients depending on $x=(x_1,\dots, x_s)$. This completes the proof of Theorem \ref{thm:A}.

We conclude this section with an example showing that the algebraic genericity assumption in Theorem \ref{thm:A} is essential.
\begin{Ex}\label{ex:1}{\rm Consider the operator  $L$ in $\R^2(x,y)$:
$$
L = \begin{pmatrix}  0 & 1 \\ 0 & x^2 - y^2          \end{pmatrix}
$$
Assume, by contradiction,  that $L=p(M)$. Then  $M$ must be $\gl$-regular, since $L$ is $\gl$-regular. Consider the sets
$$
\begin{aligned}
\mathsf{Sing}_L &= \{ (x,y)\in\R^2~|~  \mbox{the eigenvalues of $L$ collide}\} = \{ x^2 - y^2 = 0\},\\
\mathsf{Sing}_M &= \{ (x,y)\in\R^2~|~  \mbox{the eigenvalues of $M$ collide}\}. 
\end{aligned}
$$

We obviously have $\mathsf{Sing}_M \subset \mathsf{Sing}_L$.   On the other hand,  if $(x,y)\not\in \mathsf{Sing}_M$,  then $M(x,y)$ is semisimple, and therefore  $L(x,y)=p(M(x,y))$ is semisimple also, implying  $(x,y)\not\in \mathsf{Sing}_L$. Thus, $\mathsf{Sing}_L \subset \mathsf{Sing}_M$ and finally $\mathsf{Sing}_M = \mathsf{Sing}_L$. 

However, it is known that $\mathsf{Sing}_M$ is always a smooth curve in $\R^2$ (see \cite{nij3}), whereas  $\mathsf{Sing}_M$ is not.  This contradiction shows that $L$ cannot be written as a polynomial of a Nijenhuis operator.

}\end{Ex}


\section{Appendix: Haantjes operators on complex manifolds}\label{appendix}


\subsection{Haantjes operators with no real eigenvalues}\label{subsect:7.1}

The splitting theorem (Theorem \ref{thm:2})  allows us to locally split every real Haantjes operator into two blocks $L = \operatorname{diag}(L_1,L_2)$ in such a way that the eigenvalues of $L_1$ are all real at a given point $\mathsf p$,  while the eigenvalues of $L_2$ at $\mathsf p$  are all complex (i.e., have non-zero imaginary part).  The operators of the second type possess one fundamental property explained in the next theorem.

\begin{Theorem}\label{thm:6}
Let $L$ be a Haantjes operator on a real even-dimensional manifold $\mathsf M$.   Assume that $L$ has no real eigenvalues at any point of $\mathsf M$. Then there exists a unique complex structure $J$ on $\mathsf M$ satisfying the following properties 
\begin{itemize}
\item[\rm (C1)] $J$ is a polynomial in $L$ with smooth coefficients, in particular, $JL = LJ$;
\item[\rm (C2)] The real part of each eigenvalue of the operator $JL$ is negative. 
\end{itemize}   
\end{Theorem}

\begin{Remark}\label{rem:7.1} {\rm In particular, this theorem says that  $L$ preserves the complex structure and, therefore, can be treated as a $\mathbb C$-linear operator.  In complex coordinates $z_1,\dots, z_n$,   this operator is defined by a complex $n\times n$ matrix $L^{\mathbb C}$.  Moreover,  Property (C2) implies that the imaginary parts of the eigenvalues of $L^{\mathbb C}$ are all positive. This fact will be important for Theorem \ref{thm:7} below.
}\end{Remark}

\begin{proof}  In the case of a Nijenhuis operator $L$,   the existence of $J$ with the required properties was proved in \cite{nij1}.  
The proof remains basically the same:  we construct an almost algebraic structure $J$ by using purely algebraic instruments and then show that $J$ satisfies the Nirenberg--Newlander integrability condition, i.e., $\mathcal N_J = 0$.  The only difference is that now we have a seemingly weaker condition $\mathcal H_J=0$,   but  $J^2 =-\Id$ immediately implies   $\mathcal H_J=-4\,\mathcal N_J$ (see e.g. Section 4.7 in \cite{Kosmann}) so that in fact,  these conditions are equivalent.

The Haantjes case is even simpler, since we are allowed to use polynomials with functional coefficients.  So we will give an independent proof based on an explicit formula for $p(t)$ such that $J=p(L)$.        

Since the eigenvalues of $L$ can be partitioned into pairs of complex conjugate eigenvalues $\lambda^{\pm}_k= a_k \pm \mathrm{i}\, b_k$ (not necessarily distinct), $b_k >0$, $k=1,\dots, n=\frac{1}{2} \dim \mathsf M$, we can uniquely write the characteristic polynomial  of $L$  as the product of two monic polynomials  $\chi_L(t)=\chi_1(t)\chi_2(t)$ in such a way that
the roots of $\chi_1(t)$ are $\lambda_1^+,\dots, \lambda_n^+$ and the roots of $\chi_2(t)$ are $\lambda_1^-,\dots, \lambda_n^-$.
It follows from Lemma \ref{lem:5.1} that the coefficients of  $\chi_1(t)$ and $\chi_2(t)$ are smooth complex-valued functions, moreover, we notice that $\chi_2(\bar t)=\overline{\chi_1(t)}$, where `bar' denotes complex conjugation. 

Since $\chi_1(t)$ and $\chi_2(t)$ are coprime and both of degree $n$, there exist unique polynomials $q_1(t)$ and $q_2(t)$ of degree at most $n-1$ such that
\begin{equation}
\label{eq:Euclid}
\chi_1 (t) q_1(t) + \chi_2 (t) q_2(t) = 1.
\end{equation}

Moreover, since $\chi_2 (\bar t) = \overline{\chi_1(t)}$, then $q_2(\bar t) = \overline{q_1(t)}$.  Consider the polynomial $p$ of degree $2n-1$ defined as
$$
p(t) = -2\mathrm{i} \chi_1 (t) q_1(t) + \mathrm{i}
$$
or, equivalently, in view of \eqref{eq:Euclid}
$$
p(t) = 2\mathrm{i} \chi_2 (t) q_2(t) - \mathrm{i}.
$$
Notice that $p(\bar t) = \overline{p(t)}$, which means that $p$ is a polynomial with real coefficients.   The algebraic meaning of the operator  $p(L)$ is very natural: each generalised eigenspace of $L$  (as a subset of the complexified tangent space $(T_{\mathsf p} M)^{\mathbb C}$)  is $p(L)$-invariant and on the generalised eigenspace $E_{\lambda^{\pm}}$, the operator $p(L)$ acts as multiplication by $\pm \mathrm{i}$.   So $J=p(L)$ is a (real) operator on $T_{\mathsf p} M$ satisfying the condition $J^2=-\Id$. In other words, from the point of view of the whole manifold $\mathsf M$, the operator $J$ is an almost complex structure.   We also notice that the eigenvectors of $J=p(L)$ coincide with those of $L$,  and the {\it new} eigenvalues are $\lambda_{\mathrm{new},k}^{\pm}= \pm \,\mathrm{i} \cdot \lambda_k^{\pm} =  \pm \,\mathrm{i}  (a_k \pm {\mathrm i}\, b_k)= -b_k \pm \mathrm{i} a_k$. The real parts $-b_k$ of these eigenvalues are all negative as required in (C2).    

By Proposition \ref{prop:3.3}, $J$ is a Haantjes operator,  i.e. $\mathcal H_J=0$  which is equivalent, as mentioned above, to $\mathcal N_J=0$. Hence, $J$ is a complex structure on $\mathsf M$ by the Nirenberg--Newlander theorem.  \end{proof}


\subsection{Haantjes operators on complex manifolds}\label{subsect:7.2}

Let $(\mathsf M, J)$ be a complex manifold of complex dimension $n$, $z_1,\dots, z_n$ complex coordinates and $L$ an operator on $\mathsf M$ which commutes with the complex structure $J$.
This operator can be given in two different ways: as a complex $n\times n$ matrix  $L=\Bigl( L^i_j(z)\Bigr)$ and as $2n\times 2n$ real matrix  in coordinates $x_1, y_1, x_2, y_2, \dots$,  where   $z_k=x_k + \mathrm{i}\, y_k$.  To distinguish these interpretations we will use $L_{\R}$ when considering $L$ as an operator acting on the tangent space $T_{\mathsf p} \mathsf M$ treated as $2n$-dimensional vector space.   

If the components $L^i_j(z)$ are holomorphic in $z=(z_1,\dots,z_n)$,  and $L=\Bigl( L^i_j(z)\Bigr)$ satisfies the Haantjes identity
$$
\mathcal H_L (\partial_{z_i}, \partial_{z_j})=0, \quad i,j=1,\dots,n, 
$$
then we will refer to $L$ as a {\it holomorphic Haantjes operator} on a complex manifold $(\mathsf M, J)$. In this case, $L_{\R}$ is automatically a Haantjes operator on $\mathsf M$ as a $2n$-dimensional real manifold.   However,  for $L_{\R}$ to be Haantjes,  $L$ does not need to be holomorphic.  

\begin{Theorem}\label{thm:7}    
Let $L$ be a field of $\mathbb C$-linear endomorphisms on a complex manifold $(\mathsf M^n, J)$ defined, in complex coordinates $z_1,\dots,z_n$, by a complex matrix $L(z)=\Bigl( L^i_j(z)\Bigr)$ with smooth but not necessarily holomorphic entries.
\begin{itemize}
\item Assume that the `complex Haantjes torsion' of $L$ vanishes, i.e.,
\small{\begin{equation}
\label{eq:complNij}
\mathcal H_L (\partial_{z_\alpha}, \partial_{z_\beta}) =  L^2 \mathcal N_L(\partial_{z_\alpha}, \partial_{z_\beta}) + \mathcal N_L (L\partial_{z_\alpha}, L\partial_{z_\beta}) - L \mathcal N_L (L\partial_{z_\alpha}, \partial_{z_\beta}) - L \mathcal N_L (\partial_{z_\alpha}, L\partial_{z_\beta}) =0 
\end{equation}}
and  $\dfrac{\partial L}{\partial \bar z_\alpha}$ commutes with $L$ for $\alpha=1,\dots,n$.  Then $L_{\R}$ is a Haantjes operator in the real sense.

\item Conversely, if
$L_{\R}$ is Haantjes in the real sense and each eigenvalue of $L$ has positive imaginary part, then $\dfrac{\partial L}{\partial \bar z_\alpha}$ commutes with $L$ and \eqref{eq:complNij} holds.

\end{itemize} 
\end{Theorem}

Here is an example showing that the assumption about eigenvalues of $L$ in the converse statement is essential.

\begin{Ex}\label{ex:7.1}{\rm
Let  $L = \begin{pmatrix}  1 &  \bar z_1 \\ - \frac{1}{\bar z_1} & -1   \end{pmatrix}$ or, in the real interpretation, 
$$
L_{\R}=
\begin{pmatrix}
1  & 0& x_1 & y_1 \\
 0 & 1 &  - y_1 & x_1 \\
 - \frac{x_1}{x_1^2+y_1^2}&  \frac{y_1}{x_1^2+y_1^2}& -1 & 0 \\
 - \frac{y_1}{x_1^2+y_1^2}&- \frac{x_1}{x_1^2+y_1^2} & 0 & -1 
\end{pmatrix}
$$
This is a Haantjes operator, but the derivative $\frac{\partial L}{\partial \bar z_1}=\begin{pmatrix}   0 & 1 \\ \frac{1}{\bar z_1^2} & 0\end{pmatrix}$ does not commute  with $L$.
}\end{Ex}

Theorems \ref{thm:6} and \ref{thm:7} (see also Remark \ref{rem:7.1}) immediately imply the following corollary.

\begin{Corollary}\label{cor:7.1}  Let $L$ be a Haantjes operator on a real even-dimensional manifold $\mathsf M^{2n}$ with no real eigenvalues, and $J$ be the complex structure on $\mathsf M$ from Theorem \ref{thm:6}.  Then in complex coordinates $z_1,\dots, z_n$,  the operator $L$ is naturally given by a complex $n\times n$ matrix $L^{\mathbb C}= \Bigl( L^i_j(z)\Bigr)$ with  $L^i_j(z)$ being complex valued functions\footnote{In other words,  $L = (L^{\mathbb C})_\R$.}  such that 
\begin{itemize}
\item the `complex Haantjes torsion' of $L^{\mathbb C}$ vanishes, i.e.,
$\mathcal H_{L^{\mathbb C}} (\partial_{z_\alpha}, \partial_{z_\beta}) =  0$.
\item  $L^{\mathbb C}$ commutes with $L^{\mathbb C}_{\bar z_\alpha} = \Bigl( \tfrac{\partial L^i_j}{\partial \bar z_\alpha}\Bigr)$. 
\end{itemize}   

\end{Corollary}

\begin{proof}[Proof of Theorem \ref{thm:7}] We start with the first statement (which is much easier to check).

Consider the real coordinate system $x_1, \dots, x_n, y_1,\dots, y_n$, where $z_\alpha = x_\alpha + \mathrm{i}y_\alpha$
We need to check that the Haantjes torsion of  $\mathcal H_{L^{\R}}(\xi,\eta)$ vanishes for $\xi,\eta$  being any pair of basis vector fields $\partial_{x_1}, \dots, \partial_{x_n},\partial_{y_1}, \dots, \partial_{y_n}$. 

Our first observation is that these basis vector fields can be replaced with a collection of their linear combinations with constant coefficients provided the transition matrix $C$ is invertible.  Since this computation is purely algebraic  and involves only the components of $L$ and its partial derivatives w.r.t. $x_\alpha$ and $y_\alpha$,  we may also assume that the components of the transition matrix $C$ are complex numbers. This allows us to replace  
$\partial_{x_1}, \dots, \partial_{x_n},\partial_{y_1}, \dots, \partial_{y_n}$ with the vector fields
$\partial_{z_\alpha}, \partial_{\bar z_\alpha}$, $\alpha=1,\dots,n$.

Also for our computation,  we can use $\partial_{z_\alpha}, \partial_{\bar z_\alpha}$ as a new basis in which the matrix of $L_\R$ has the following block diagonal form:
\begin{equation}
\label{eq:Lreal}
L_\R = \begin{pmatrix} L  & 0 \\
0 & \bar L\end{pmatrix} 
\end{equation}

Thus, the vanishing of the Haantjes torsion for $L_\R$  means that for all $i,j=1,\dots,n$ we have
$$
\mathcal H_{L_\R} (\partial_{z_i}, \partial_{z_j})=0,\quad
\mathcal H_{L_\R} (\partial_{\bar z_i}, \partial_{z_j})=0 \quad\mbox{and} \quad
\mathcal H_{L_\R} (\partial_{\bar z_i}, \partial_{\bar z_j})=0. 
$$

Due to the above block-diagonal structure,  the first condition is equivalent to \eqref{eq:complNij}. 
The third condition is obtained from the first one by complex conjugation and, therefore, also follows from \eqref{eq:complNij}.  It remains to  
understand the meaning of the second condition.  We claim that  $\mathcal H_{L_\R} (\partial_{z_i}, \partial_{\bar z_j})=0$ follows from the condition that $L_{\bar z_\alpha}$ commutes with $L$ and, conversely, the commutativity relation $L_{\bar z_\alpha} L = L \, L_{\bar z_\alpha}$ follows from the vanishing of $\mathcal H_{L_\R} (\partial_{z_i}, \partial_{\bar z_j})$  under the additional assumption that $L$ and $\bar L$ have no common eigenvalues (this assumption is guaranteed by the fact that the imaginary part of each eigenvalue of $L$ is positive).

To justify this statement, we apply that part of the proof of Theorem \ref{thm:2}, where we discussed
necessary and sufficient conditions for a block-diagonal operator  $L=\begin{pmatrix}  A(u,v) & 0 \\ 0 & B(u,v) \end{pmatrix}$ to be Haantjes. Namely,  it was shown that if  $A_{v^\alpha} A = A \, A_{v^\alpha}$ and $B_{u^\beta} B = B \, B_{u^\beta}$ for all $\alpha,\beta$, then 
$\mathcal H_L(\partial_{u^i}, \partial_{v^j})=0$, and, conversely, under additional condition that $A$ and $B$ have no common eigenvalues, 
$\mathcal H_L(\partial_{u^i}, \partial_{v^j})=0$ for all $i,j$ implies $A_{v^\alpha} A = A \, A_{v^\alpha}$ and $B_{u^\beta} B = B \, B_{u^\beta}$.

In our current situation, the only difference is that now we work with a complex basis $\partial_{z_\alpha}, \partial_{\bar z_\alpha}$ of the complexified tangent space instead of a real basis $\partial_{u^i}, \partial_{v^j}$.  However, the proof in the real case (see Lemma \ref{lem:5.2} and discussion just after the Lemma) was purely algebraic, i.e., the conclusion was derived from the polynomial relations between the components of $L$ and their partial derivatives given by the Haantjes condition $\mathcal H_L = 0$.  Thus,  these arguments still work for basis vectors with complex coefficients and  $L_\R = \begin{pmatrix}  L & 0 \\ 0 & \bar L \end{pmatrix}$, i.e., for $A=L$, $B=\bar L$.  It remains to notice that in this setting, condition (2) means exactly that $L$ commutes with $L_{\bar z_\alpha}$  and $\bar L$ commutes with $\bar L_{z_\alpha}$ (which is the same due to duality between $z$ and $\bar z$).  This completes the proof. 
\end{proof}

\subsection{Complex operators \texorpdfstring{$L$}{L} such that
\texorpdfstring{$[L,L_{\bar z_i}]=0$}{[L, Lbar zi] = 0}}\label{subsect:7.3}

Let $L(z)$,  $z=(z_1,\dots, z_k)\in \mathbb C^k$ be a complex $n\times n$ matrix depending on complex parameters. We assume that the entries $L^i_j(z)$ are $C^\infty$-smooth in $z$, but not necessarily holomorphic. 
In this appendix we prove several properties of $L(z)$ under the two conditions:
\begin{itemize}
\item $L(z)$ is $\gl$-regular and algebraically generic;

\item $L(z)$ commutes with $L_{\bar z_i}=\frac{\partial}{\partial \bar z_i} L(z)$, i.e.
$
L\,L_{\bar z_i} = L_{\bar z_i} L
$  for all $i=1,\dots,k$. 
\end{itemize}

For the purposes of our paper, $L(z)$ should be understood as a field of $\mathbb C$-linear (but not necessarily holomorphic) endomorphisms on a complex manifold.  But in this appendix,  the tensorial nature of $L(z)$ does not play any role and we may think of $z$ as just a parameter.

\begin{Theorem}\label{thm:8}
Let $L(z)$ be $\gl$-regular, algebraically generic and commute with $L_{\bar z_i}$, $i=1,\dots,k$.   If $L$ is semisimple, then there exists an eigenbasis $e_1(z), \dots, e_n(z)$ whose vectors $e_\beta(z)$ are holomorphic in $z$.  
More generally, if $\lambda=\lambda(z)$ is an eigenvalue of $L(z)$ of multiplicity $m$, then there exist vectors $e_1(z), \dots, e_m(z)$ holomorphic in $z$ such that $\operatorname{Span}(e_1(z),\dots, e_s(z))=\operatorname{Ker}(L-\lambda\, \Id)^s$, $s=1,\dots, m$.
\end{Theorem}

This statement immediately implies the following corollary. 

\begin{Corollary}\label{cor:7.2}
In the assumptions of Theorem \ref{thm:8},  let $\lambda_1(z), \dots, \lambda_r(z)$ denote the distinct eigenvalues of $L(z)$ and $m_1,\dots, m_r$ be their multiplicities, $\sum m_\alpha = n$. Then there exists a holomorphic (in $z$) transition matrix $C(z)$ such that 
$\widetilde L(z) = C^{-1} (z) L(z) C(z)$ is block diagonal: 
$$
\widetilde L(z) = \operatorname{diag} \bigl(A_1(z), \dots, A_r(z)\bigr),
$$ 
where $A_\alpha$-block is an $m_\alpha\times m_\alpha$ upper triangular matrix with $\lambda_\alpha(z)$ on the diagonal.
\end{Corollary}

\begin{proof}[Proof of Theorem \ref{thm:8}]

We first notice that the properties of $L(z)$ from the assumptions of Theorem \ref{thm:8} are preserved under similarity transforms $L(z) \mapsto C^{-1}(z)L(z)C(z)$, where $C(z)$ is holomorphic in $z$ and under addition of a scalar matrix  $L(z) \mapsto L(z)+ c(z)\Id$ where $c(z)$ is smooth.    

We will prove the second, more general statement concerning eigenvalues with multiplicities. We need to show that the vectors $e_1,\dots, e_m$  (where $m$ is a multiplicity of $\lambda$)  can be chosen in such a way that $\frac{\partial}{\partial \bar z_i} e_s(z) = 0$ for all $i=1,\dots, k$ and $s=1,\dots, m$.  The proof is the same for all variables $z_i$.  For this reason,  to simplify notation we will write  $\frac{\partial}{\partial \bar z}$ and $\frac{\partial}{\partial  z}$ instead of $\frac{\partial}{\partial \bar z_i}$ and $\frac{\partial}{\partial  z_i}$ (formally this means that we work with only one parameter $z=z_i$, keeping the others fixed). 

We will need the following lemma.

\begin{Lemma}\label{lem:7.1}
Let $L(z)$ satisfy the assumptions of Theorem \ref{thm:8}. If $e(z)$ is an eigenvector of $L(z)$ with the eigenvalue $\lambda(z)\equiv 0$,  then $L_{\bar z} e = 0$. 
\end{Lemma}

\begin{proof}  We have $L(z) e(z)=0$ identically for all $z$. Then
$$
0=\frac{\partial}{\partial \bar z} \bigl( L(z) e(z) \bigr) = L_{\bar z} e + L e_{\bar z}.
$$
Hence, $L^2 e_{\bar z} = - L (L_{\bar z} e) = - L_{\bar z} Le=0$.  This implies that $e_{\bar z}$ is a generalised eigenvector of height 2 and  $L_{\bar z} e$ is an eigenvector.

Without loss of generality we may assume that we work in a neighborhood of the origin, i.e., $z_0=0$.  Consider the formal expansion of $L(z)$ in powers of $z$ and $\bar z$ and write it in the form
$$
L(z) \simeq A_0 + z A_1 + z^2 A_2 + z^3 A_3 + \dots 
$$
where $A_i$ is a formal matrix power series in $\bar z$. In particular,
$$
A_0 = L(0) + \bar z\, L_{\bar z}(0) + \frac{\bar z^2}{2!} \, L_{\bar z\bar z} + \frac{\bar z^3}{3!} \,  L_{\bar z\bar z\bar z} + \dots 
$$

It can be easily checked that the $\bar z$-derivative of all orders commute with $L$.   In other words, all the matrices from this series belong to the centraliser of $L(0)$.  Since $L(0)$ is $\gl$-regular,  each of these matrices is a polynomial in $L(0)$.  Hence,  $A_0$ can be understood as a polynomial in $L(0)$ with coefficients $f_i=f_i(\bar z)$ being formal power series in $\bar z$:
$$
A_0 = f_0 \Id  + f_1 L(0) + f_2 L^2(0) + \dots + f_{n-1} L^{n-1}(0) 
$$
Without loss of generality we may assume that $L(0)$ is a Jordan matrix whose first Jordan $m\times m$ block is related to the zero eigenvalue. Hence  $A_0$ has similar block diagonal structure and its first block is an upper triangular Toeplitz matrix of the form
$$
\begin{pmatrix}  
f_0 & f_1 &  f_2 & \dots & f_{n-1} \\
 &f_0 & f_1 & \ddots & \vdots \\
 & & f_0 &\ddots & f_2 \\
 & & & \ddots & f_1 \\
 & & & & f_0\\
 \end{pmatrix}, \qquad f_i = f_i(\bar z).
$$
Similarly, the first block of (the constant matrix) $L_{\bar z} (0)$ is 
$$
\begin{pmatrix}  
a_0 & a_1 &  a_2 & \dots & a_{n-1} \\
 &a_0 & a_1 & \ddots & \vdots \\
 & & a_0 &\ddots & a_2 \\
 & & & \ddots & a_1 \\
 & & & & a_0\\
 \end{pmatrix},
$$
where the first terms of the formal power series $f_0$ are as follows $f_0 = 0 + a_0 \bar z +\dots$.  It remains to notice that the number $a_0$ must be equal to zero.  The point is that $\lambda=0$ is an eigenvalue of $L(z)$ for all $z$. Therefore, the (formal) determinant of $A_0$ identically vanishes.  However,  $\det A_0$ is the product of 
$f_0^m(\bar z)$ and the non-zero determinants of the remaining blocks of $A_0$.   
Thus, $\det A_0=0$ implies $f_0(\bar z)=0$ and $a_0=0$. 

It remains to notice that in the chosen basis, the eigenvector $e(0)$ is the first basis vector $(1,0,\dots,0)^\top$. Since $a_0=0$, we have $L_{\bar z} (0) e(0)=0$, as stated. \end{proof}

Then for arbitrary $z$,  we uniquely choose the (normalised) generalised eigenvectors in such a way that
$$
\begin{aligned}
e_1(z) &= (1, * , * , * \dots )\\
e_2(z) &= (0,1, * , *, \dots)\\
e_3(z) &= (0, 0 , 1, * ,\dots) \\
& \mbox{etc.}
\end{aligned}
$$
where the stars $*$ denote some functions in $z$.
We claim that the vectors chosen in this way are all holomorphic in $z$, i.e., the $\bar z$-derivatives of all of their components (denoted by $*$) are equal to zero at any point $z$.

Let us prove this for $e_1(z)$ which is an eigenvector of $L(z)$.  Using Lemma \ref{lem:7.1},
we have 
$$
0 = \frac{\partial}{\partial \bar z} (L(z) e_1(z)) = L_{\bar z} (z) e_1(z) + L \tfrac{\partial}{\partial \bar z} e_1(z) = 0 + L \tfrac{\partial}{\partial \bar z} e_1(z).
$$
In other words, $\tfrac{\partial}{\partial \bar z} e_1(z)$ is an eigenvector of $L$, which must be proportional to $e_1(z)$.  But this is only possible if $\tfrac{\partial}{\partial \bar z} e_1(z)=0$, as required.

Similarly, for $e_s(z)$ we have   (using the fact that $L^{s-1} e_s(z)$ is a non-zero eigenvector of $L(z)$ and Lemma  \ref{lem:7.1})
$$
0 = \frac{\partial}{\partial \bar z} (L^s(z) e_s(z)) = sL_{\bar z} (z)  L^{s-1} e_s(z) + L^s \tfrac{\partial}{\partial \bar z} e_s(z) = 0 + L^s \tfrac{\partial}{\partial \bar z} e_s(z).
$$

This means that $\tfrac{\partial}{\partial \bar z} e_s(z)$ is a generalised eigenvector for $L$ of height $s$.   But the space of such eigenvectors is spanned by $e_1,\dots, e_s$ and $\tfrac{\partial}{\partial \bar z} e_s(z)$ does not belong to this space unless $\tfrac{\partial}{\partial \bar z} e_s(z)=0$.  Thus all $e_s(z)$ are holomorphic as stated. This completes the proof of Theorem \ref{thm:8}.
\end{proof}

From Theorem \ref{thm:8} we can also derive the following two properties.

\begin{Corollary}\label{cor:7.3}
Let $L(z)$ satisfy the assumptions of Theorem \ref{thm:8}.  Then $L(z)$ is a polynomial of degree $n-1$ of a holomorphic operator $\widetilde L(z)$ with coefficients that are not necessarily holomorphic.
\end{Corollary}

\begin{Corollary}\label{cor:7.4}
Let $L(z)$ satisfy the assumptions of Theorem \ref{thm:8}.  Then there exists a holomorphic (in $z$) transition matrix $C(z)$ such that
$$
C^{-1}(z) L(z)  C(z) = T(z),
$$
where $T$ is a blockwise Toeplitz upper triangular matrix.
\end{Corollary}

\begin{proof}[Proof of Corollaries \ref{cor:7.3} and \ref{cor:7.4}]

We first prove Corollary \ref{cor:7.3}.    Notice that the assumptions on $L$ and the property we need to prove are invariant w.r.t. similarity transformations with holomorphic transition matrices $C(z)$.  Hence,  using Corollary \ref{cor:7.2} we may assume that $L$ is a  block diagonal matrix with upper triangular blocks (having different eigenvalues).  We also notice that it is sufficient to prove the statement for each block separately.   Moreover,  if we consider an individual block, we may assume that this block is nilpotent.  

Thus, we need to check that a strictly upper triangular matrix   
\begin{equation}\label{L4}
L = \begin{pmatrix}
0 & b_1 & * &\dots  & * \\
& 0 & b_2 & \ddots & \vdots \\
& & 0& \ddots & * \\
& & &\ddots & b_{n-1} \\
& & & & 0 \\
\end{pmatrix}, \qquad b_i(z)\ne 0
\end{equation}
satisfying the condition that $[L, L_{\bar z}]=0$, is a polynomial of a holomorphic matrix $\widetilde L$. 

It is easy to see that a $\gl$-regular matrix $L$ can be written as a polynomial of a certain matrix $M$ if and only if the centralisers of $L$ and $M$ coincide  (here we use the fact that $L$ is similar to a Jordan block or, more generally, $L$ is $\gl$-regular).

Thus, we can equivalently prove that the centraliser of $L$ contains a holomorphic $\gl$-regular operator.  The following algebraic fact seems to be very helpful.

\begin{Lemma}
For any matrix \eqref{L4} there exists a unique matrix $\widetilde L$ of the form 
\begin{equation}
\label{L5}
\widetilde L = 
\begin{pmatrix}
0 & 1 & 0 & \dots & 0 \\
&0 & c_2 & \dots  & * \\
& & 0 & \ddots& \vdots \\
& & & \ddots & c_{n-1} \\
& & & & 0 \\
\end{pmatrix}, \qquad c_i\ne 0,
\end{equation}
which commutes with $L$.
\end{Lemma}
\begin{proof}
This fact is purely algebraic.  It is sufficient to explicitly find $\widetilde L$  of the form \eqref{L5} from the matrix equation $[L, \widetilde L]=0$  with $L$ given by \eqref{L4}.  This linear equation system is triangular and can be solved step-by-step:  first we find $c_2,\dots, c_{n-1}$, then the element located on the next  diagonal line and so on.  Alternatively, we can find a unique polynomial such that $\widetilde L= p(L)$.   \end{proof}

Next we notice that  $[L, L_{\bar z}]=0$ implies $[\widetilde L,  \widetilde L_{\bar z}]=0$.

Finally, let us check that every operator \eqref{L5} satisfying the condition  $[\widetilde L,  \widetilde L_{\bar z}]=0$ is automatically holomorphic.  The latter condition means that $\widetilde L_{\bar z}=0$.   This fact can be justified as follows. First notice that $[\widetilde L,  \widetilde L_{\bar z}]=0$ implies that $\widetilde L_{\bar z} = q(\widetilde L)$, where $q(\cdot)$ is a certain polynomial, i.e.
$$
\widetilde L_{\bar z}  =  d_0 \Id + d_1 \widetilde L + \dots + d_{n-1}\widetilde L^{n-1}.
$$
 where $\widetilde L^m$ is a matrix of the form 
$$
\widetilde L^m = \begin{pmatrix}
0 & \dots & 0  &  c  & \dots  & * \\
  & \ddots & \vdots  & 0  & \ddots & \vdots \\
  & & 0 &\vdots &\ddots  & *  \\
  & & & 0 && 0  \\
  & & & & \ddots &\vdots  \\
  & & & & & 0 \\
\end{pmatrix},   
$$
with   $c = \left(\widetilde L^m\right)_{m+1}^1 \ne 0$.   In particular, the first rows of the matrices $\Id,\widetilde L, \dots , \widetilde L^{n-1}$ are linearly independent (as row-vectors of length $n$).
It remains to notice that the first row of $\widetilde L_{\bar z}$ is zero. Therefore all the coefficients $d_0, \dots, d_{n-1}$ vanish simultaneously and $\widetilde L_{\bar z}  = 0$, as required.  Thus, $\widetilde L$ is holomorphic.
This completes the proof of Corollary \ref{cor:7.3}.

Corollary \ref{cor:7.4}  easily follows from Corollary \ref{cor:7.3}.  Indeed, $L = p(\widetilde L)$, where $\widetilde L$ is holomorphic and nilpotent.  Then for $\widetilde L$ we can find a holomorphic transition matrix $C(z)$ such that 
$\widetilde L = C(z) J C^{-1}(z)$, where $J$ is the standard nilpotent Jordan block  (constant). Then 
$L = p(C(z) J C^{-1}(z)) = C(z) p(J) C^{-1}(z)$, where $T(z) = p(J)$ is Toeplitz as required. 
\end{proof}

Finally,  we combine Theorems \ref{thm:6}, \ref{thm:7} and Corollary \ref{cor:7.4} to derive one more property of  Haantjes operators with complex eigenvalues, which is crucial in the proof of Theorem \ref{thm:1b}.

\begin{Corollary}\label{cor:7.5}
Let $L$ be a Haantjes operator which is similar to a real Jordan block related to a pair of complex conjugate eigenvalues $\lambda^{\pm} = a \pm \mathrm{i} \, b$.  Then there exists a complex structure $J$ such that in any complex coordinate system $z_1,\dots, z_n$, the operator $L$ is given by an $n\times n$ complex matrix $L^{\mathbb C}$ which can be locally written in the form $L^{\mathbb C} = p(\widetilde L)$, where $\widetilde L$ is a holomorphic Haantjes operator and $p(\cdot )$ is a certain polynomial whose coefficients are complex-valued smooth functions of $z_1,\dots, z_n$, not necessarily holomorphic. More generally, this conclusion holds true locally for any $\gl$-regular algebraically generic Haantjes operator $L$ with no real eigenvalues.

\end{Corollary}

\subsection*{Acknowledgments.}  A.\,B. was  supported by the Ministry of Science and Higher Education of the Republic of Kazakhstan (grant No. AP23483476),   V.\,M. by  the DFG (project  529233771) and the ARC Discovery Programme DP210100951, and A.\,K. by the Ministry of Education and Science of the Russian Federation as part of the program of the Moscow Center for Fundamental and Applied Mathematics (Agreement  075-15-2025-345).

  Data sharing is not applicable to this article. The authors declare no conflicts of interest.

\printbibliography

\end{document}